\documentclass[preprint,12pt]{elsarticle}
\usepackage{amssymb,amsmath,graphicx,amsthm}
\usepackage{commath}
\usepackage{bm}
\usepackage[dvipsnames]{xcolor}
\usepackage{subcaption}
\usepackage{placeins}
\usepackage{hyperref}
\theoremstyle{remark}
\newtheorem{remark}{Remark}

\theoremstyle{plain}
\newtheorem{theorem}{Theorem}
\newtheorem{lemma}[theorem]{Lemma}

\theoremstyle{definition}
\newtheorem{assumption}{Assumption}

\begin{document}
\begin{frontmatter}
\title{I-MCHM: Interface Multicontinuum Homogenization for Multiscale Elliptic Problems}
\author[cityu]{Wing Tat Leung\corref{cor1}}
\ead{wtleung27@cityu.edu.hk}
\author[sit]{Zhihang Xu}
\ead{xuzhihang@sit.edu.cn}
\cortext[cor1]{Corresponding author}
\address[cityu]{Department of Mathematics, City University of Hong Kong, Hong Kong}
\address[sit]{Department of Mathematics, Shanghai Institute of Technology, Shanghai 201418, China}
\begin{abstract}
We introduce an interface multicontinuum homogenization method (I-MCHM) for
high-contrast elliptic problems whose subdomains may have distinct microscopic
patterns and unequal numbers of continua.  Multicontinuum models upscale
microscopic fields to macroscopic continuum quantities; here the continua are
the high- and low-permeability regions, and the macroscopic variables are the
corresponding local averages of the fine-scale solution on observing cells of
size $H_\epsilon$.  Although the fine-scale solution is continuous across a
material interface, these macroscopic continuum fields need not coincide on the
interface.  Moreover, when adjacent subdomains retain different numbers of
continua, a one-to-one coarse transmission condition cannot be defined.
Standard subdomain-wise multicontinuum homogenization therefore determines the
bulk equations but leaves the interaction among these coarse variables
unspecified.  I-MCHM closes this gap by constructing two-sided constrained local
bases on interface neighborhoods and assembling a coarse bilinear form with
bulk cut-cell integrals and interface-segment corrections over the subdomain
adjacency graph, without prescribing a pointwise transmission law.  Numerical
experiments on straight, curved, and triple-junction geometries, including
continuum-coupling ablations and unequal continuum counts, demonstrate the
accuracy of the resulting upscaled model.
\end{abstract}
\end{frontmatter}
\section{Introduction}

Multiscale elliptic models of composites and porous media often contain
fine-scale oscillations, high contrast, and abrupt material changes, making
direct fine-grid resolution expensive~\cite{hou1997multiscale,e2003heterogeneous,
efendiev2009multiscale}.  Classical and numerical homogenization replace these
details by effective equations obtained from local representative problems
~\cite{bensoussan1978asymptotic,allaire1992homogenization,
engquist2008asymptotic,abdulle2012heterogeneous}.  Complementary
discretization-based methods, including MsFEM, LOD, GMsFEM, CEM-GMsFEM, and
NLMC, construct mesh-dependent coarse spaces instead
~\cite{hou1997multiscale,henning2014localized,efendiev2013generalized,
chung2018non}.  The present work seeks a homogenized continuum description
whose effective coefficients can subsequently be used with different coarse
discretizations.

The Multicontinuum Homogenization Method
(MCHM)~\cite{efendiev2022multicontinuum,chung2024multicontinuum}
derives coupled macroscopic equations for several continuum quantities.  These
quantities may represent physical compartments, spectral components, or other
microscopic features~\cite{barenblatt1960basic,arbogast1990derivation}.  We
adopt conductivity-based continua: each unknown is the average of the
fine-scale solution over the high- or low-permeability part of an observing
cell of size $H_\epsilon$.  Constrained local bases reproduce these moments,
and their insertion into the fine variational form yields the effective
coefficients.

Material interfaces that separate distinct microstructures arise in composites,
geological layers, and fractured media~\cite{milton2022theory,bear2013dynamics,
berre2018flow}.  Related computational work includes fitted and unfitted
finite-element treatments of elliptic interface transmission
conditions~\cite{chen1998finite,leveque1994immersed}, multiscale finite-element
constructions for high-contrast interfaces that need not resolve the jump set on
the coarse mesh~\cite{chu2010multiscale,hou2018iteratively}, and spectral or
nonlocal multicontinuum enrichments for channelized high-conductivity
features~\cite{efendiev2011multiscale,chung2018non}.  Homogenization analyses
with interfacial jumps likewise lead to continuum models with transmission
corrections~\cite{bunoiu2019upscaling}.  The central difficulty for
multicontinuum averages is that continuity of the fine-scale solution
$u_\epsilon$ does not imply continuity of its continuum averages.  In
particular, it does not imply
\[
 U_j^{(k)}\big|_{\Gamma_{k\ell}}
 =
 U_{\pi(j)}^{(\ell)}\big|_{\Gamma_{k\ell}}
\]
for a prescribed pairing $\pi$; no such pairing exists when $L_k\ne L_\ell$.
Independent bulk MCHM models therefore leave the cross-interface interaction
unspecified.  We obtain that interaction from a joint two-sided local problem
rather than by postulating a coarse transmission law.

We develop I-MCHM for problems with subdomain-dependent microstructures and
continuum counts.  Bulk neighborhoods employ standard MCHM.  On each interface
neighborhood, saddle-point problems construct bases that reproduce independent
moments from both sides.  A bulk--interface decomposition interprets the
resulting localization, and the local energies assemble into a coarse bilinear
form with bulk cut-cell integrals and interface-segment corrections over the
subdomain adjacency graph.

The main contributions of this work are as follows:
\begin{enumerate}
    \item two-sided constrained bases for subdomain-wise continuum fields that
    may be discontinuous or have unequal counts across an interface;
    \item a symmetric coarse bilinear form with bulk cut-cell integrals and
    interface-segment corrections, requiring no prescribed pointwise
    transmission law;
    \item a bulk--interface basis decomposition together with discrete
    well-posedness in the reconstruction-energy norm; and
    \item numerical experiments on straight, curved, and three-subdomain
    geometries, including continuum-coupling ablations and unequal continuum
    counts.
\end{enumerate}

The remainder of the paper is organized as follows.
Section~\ref{sec:problem_setup} states the fine-scale interface problem and the
scale hierarchy.
Section~\ref{sec:review_mchm} reviews bulk MCHM in the absence of material
interfaces.
Section~\ref{sec:interface-mchm} develops I-MCHM: two-sided bases, bulk
cut-cell and interface-segment coefficients, the global weak form, and discrete
stability.
Section~\ref{sec:nume} reports numerical experiments and Section~\ref{sec:conclu} concludes.

\section{Problem setup}
\label{sec:problem_setup}
We introduce the fine-scale problem and the interface configuration.
Let $\Omega\subset\mathbb{R}^d$ ($d=2,3$) be a bounded computational domain with
a sufficiently regular boundary.  We consider the fine-scale elliptic problem
\begin{equation}
    -\nabla\cdot(\kappa_{\epsilon}(x)\nabla u_\epsilon)=f
    \quad \text{in } \Omega,
    \label{eq:fine_problem}
\end{equation}
subject to the homogeneous Dirichlet boundary condition
\begin{equation*}
    u_{\epsilon}= 0 \quad \text{on } \partial \Omega\,.
\end{equation*}
Here $u_{\epsilon}$ denotes the fine-scale solution, $f$ is a given source, and
$\kappa_{\epsilon}$ is a heterogeneous coefficient that may contain highly
oscillatory or high-contrast structures.
The corresponding weak formulation seeks $u_\epsilon\in H_0^1(\Omega)$ such that
\begin{equation}
    a_\epsilon(u_\epsilon,v) = (f,v),
    \qquad \forall v\in H_0^1(\Omega),
    \label{eq:fine_weak}
\end{equation}
where
\begin{equation}
    a_\epsilon(u,v)
    :=
    \int_\Omega \kappa_{\epsilon}(x)\nabla u\cdot\nabla v\,\dif x,
    \qquad
    (f,v):=\int_\Omega f v\,\dif x.
\end{equation}
More generally, for any measurable subdomain $D\subset\Omega$, we define the
localized bilinear form
\begin{equation}
a_{\epsilon,D}(u,v)
:=
\int_D
\kappa_\epsilon(x)\nabla u(x)\cdot\nabla v(x)\,\dif x.
\label{eq:local_inner}
\end{equation}

We upscale $u_\epsilon$ to continuum-wise averages over the high- and
low-permeability parts of observing cells.  Below, $x$ denotes the macroscopic
position and $y$ the local microscopic coordinate.

We assume that $\Omega$ is partitioned into $N_\Omega$ non-overlapping
Lipschitz subdomains $\{\Omega_k\}_{k=1}^{N_\Omega}$ such that
\[
\overline{\Omega} = \bigcup_{k=1}^{N_\Omega} \overline{\Omega}_k,
\qquad
\Omega_k\cap\Omega_\ell=\emptyset\quad(k\ne\ell).
\]
For every adjacent pair of subdomains, define
\[
\Gamma_{k\ell}:=\partial\Omega_k\cap\partial\Omega_\ell,
\qquad
\mathcal E:=\{(k,\ell):k<\ell,\ |\Gamma_{k\ell}|_{d-1}>0\},
\]
and let $\Gamma=\bigcup_{(k,\ell)\in\mathcal E}\Gamma_{k\ell}$.
Subdomain $\Omega_k$ may contain $L_k$ conductivity-based continua, with
$H_\epsilon$-cell averages $U_j^{(k)}$.  Since the fine-scale solution belongs to
$H_0^1(\Omega)$, the transmission conditions
\[
[\![u_\epsilon]\!]_{\Gamma_{k\ell}}:=u_\epsilon\big|_{\Omega_k}-u_\epsilon\big|_{\Omega_\ell}=0,
\qquad
[\![\kappa_\epsilon\nabla u_\epsilon\cdot n]\!]_{\Gamma_{k\ell}}=0
\]
hold weakly, where $n$ is a unit normal.  These conditions are not imposed
componentwise on the coarse variables; Section~\ref{sec:interface-mchm}
instead defines their interaction variationally.

\subsection{Meshes, RVEs, and oversampling}
\label{sec:mesh}
We distinguish the microscopic scale $\epsilon$, the observing-cell size
$H_\epsilon$, and the coarse scale $H$, with
$\epsilon\le H_\epsilon<H$.  Let $\mathcal T_{H_\epsilon}$ and
$\mathcal T_H$ be shape-regular, interface-resolving partitions; boundary
fragments are merged with a nearest interior cell when needed.  The numerical
experiments use nested Cartesian partitions.  Let $K_H(x)\in\mathcal T_H$
denote a coarse block centered at $x$.
When the fine-scale structure varies on a scale much smaller than $H$, the
microscopic response in $K_H(x)$ can be approximated by that in a smaller
representative volume element
$\omega(K_H(x))\subset K_H(x)$, taken as a union of one or more
$H_\epsilon$-cells of $\mathcal{T}_{H_\epsilon}^x$ that share the same continuum
pattern as $K_H(x)$ (typically a single observing cell or a small contiguous
cluster of observing cells about $x$).
No global periodicity of $\kappa_\epsilon$ is assumed: the RVE is a local
sampling region whose average energy density is taken as representative of the
coarse block,
\[
    \frac{1}{|\omega(K_H(x))|}
    \int_{\omega(K_H(x))}
    \kappa_{\epsilon}(x)\nabla u\cdot\nabla v\dif x
    \approx
    \frac{1}{|K_H(x)|}
    \int_{K_H(x)}
    \kappa_{\epsilon}(x)\nabla u\cdot\nabla v\dif x.
\]
Thus local bilinear forms over the coarse block may be approximated by scaled
integrals over the RVE.

The oversampled RVE is obtained by adjoining $s_{\rm os}$ layers of
$H_\epsilon$-cells:
\[
\omega^+(K_H(x))
:=
\operatorname{oversample}_{s_{\rm os}}\omega(K_H(x)).
\]
In the bulk and interface constructions below we write $\omega(K)$ and
$\omega^+(K)$ for a generic coarse neighborhood $K$.
Oversampling reduces artificial local-boundary effects.

\begin{assumption}[Modeling regime]
\label{ass:modeling}
The construction is used under the following hypotheses.
\begin{enumerate}
\item The scales satisfy $\epsilon\leq H_\epsilon<H$, with sufficiently many
observing cells in each local patch to identify every retained continuum.
\item The coefficient is measurable and strictly positive for each fixed
$\epsilon$.  High contrast may depend on $\epsilon$, but every active
continuum-cell intersection has positive measure.
\item The continuum averages are defined on observing cells of size
$H_\epsilon$ and are regarded as samples of macroscopic fields that vary
slowly on the coarse scale $H$; first-order local Taylor expansions of these
fields are therefore meaningful on an RVE of size comparable to $H$.
\item Each bulk RVE and each two-sided interface RVE is representative of the
coarse neighborhood to which its energy density is assigned.  Global
periodicity of $\kappa_\epsilon$ is not required.  The prescribed
oversampling depth $s_{\mathrm{os}}$ is sufficient to suppress the dominant
artificial boundary effect of the local Dirichlet condition.
\end{enumerate}
\end{assumption}
These are modeling assumptions used to define the reduced system.  Localization
of constrained energy-minimizing bases follows from the CEM and NLMC
analysis~\cite{zhao2020analysis}.  We invoke that decay mechanism below under
its standard stable-decomposition and continuum-coverage conditions.

\paragraph{Notation.}
$\omega(\nu)$ and $\omega^+(\nu)$ denote the RVE and oversampled RVE of a
bulk or interface neighborhood $\nu$.  The reconstructions are $R_K$ (bulk),
$R_E$ (interface), and $R_\nu$ generically.  Continuum fields are collected in
$\bm U_H$, while $\mathbf U_H$ denotes their discrete coefficient vector.
The bulk construction recalled in the next section is the interior building
block of I-MCHM; the interface closure is developed in
Section~\ref{sec:interface-mchm}.

\section{Review of bulk MCHM}
\label{sec:review_mchm}

In this section we review the standard multicontinuum homogenization method
(MCHM)~\cite{efendiev2022multicontinuum,chung2024multicontinuum}
in the absence of material interfaces.  The resulting bulk coefficients and
local reconstructions are reused on interior coarse blocks of the global
I-MCHM system.

\subsection{Continua and their characteristic functions}
The continuum choice is problem-dependent
~\cite{barenblatt1960basic,chung2024multicontinuum}.  Here each continuum is a
high- or low-permeability region, identified by a characteristic function, and
its macroscopic variable is the corresponding $H_\epsilon$-cell average.

In each coarse block $K\in\mathcal T_H^x$, assume that the medium contains
$L$ such continua.  We denote the characteristic function of continuum $i$ by
$\psi_i$ (with a superscript indicating the cell when needed).  In this work,
$\psi_i$ is the indicator of the high- or low-permeability region.  Let
$\mathcal I_K$ index the observing cells contained in the oversampled local
region.  For each $q\in\mathcal I_K$, we write $\psi_{j,q}(y)$ for the
characteristic function of continuum $j$ restricted to observing cell $q$.

\subsection{Bulk MCHM construction}
For an observing cell
$K_{H_\epsilon}(x)\in\mathcal T_{H_\epsilon}^x$ centered at $x$, we set
\[
U_i(x)
:=
\frac{\int_{K_{H_\epsilon}(x)}u_\epsilon(y)\psi_i^{K_{H_\epsilon}(x)}(y)\,\dif y}
{\int_{K_{H_\epsilon}(x)}\psi_i^{K_{H_\epsilon}(x)}(y)\,\dif y}.
\]
Treating these averages as samples of slowly varying macroscopic fields, we
employ the local expansion
\[
u_\epsilon(y)
\approx
\sum_{i=1}^L
\left(
U_i(x_K)\eta_i^{K}(y)
+
\sum_{m=1}^d
\partial_mU_i(x_K)\eta_i^{K,(m)}(y)
\right),
\qquad y\in\omega(K),
\]
on a coarse neighborhood $K$ with center $x_K$.  For a local domain $D$, we
write
\[
V(D):=H_0^1(D),
\qquad
a_{\epsilon,D}(w,v)
:=
\int_D\kappa_\epsilon(y)\nabla w\cdot\nabla v\,\dif y.
\]
The oversampled value basis $\eta_i^{K,+}\in V(\omega^+(K))$ minimizes
$\tfrac12a_{\epsilon,\omega^+(K)}(\eta,\eta)$ subject to the $L^2$ moment
constraints
\[
\int_{\omega^+(K)}
\eta(y)\psi_{j,q}(y)\,\dif y
=
\delta_{ij}
\int_{\omega^+(K)}
\psi_{j,q}(y)\,\dif y,
\qquad
j=1,\ldots,L,\quad q\in\mathcal I_K.
\]
Likewise, for $m=1,\ldots,d$, the first-order basis
$\eta_i^{K,+,(m)}\in V(\omega^+(K))$ minimizes the same energy subject to
\[
\int_{\omega^+(K)}
\eta(y)\psi_{j,q}(y)\,\dif y
=
\delta_{ij}
\int_{\omega^+(K)}
(y_m-x_{K,m})\psi_{j,q}(y)\,\dif y,
\]
where $x_K=(x_{K,1},\ldots,x_{K,d})$.
The basis functions used in the reconstruction are the restrictions
$\eta_i^K=\eta_i^{K,+}|_{\omega(K)}$ and
$\eta_i^{K,(m)}=\eta_i^{K,+,(m)}|_{\omega(K)}$.

The resulting bulk reconstruction is
\begin{equation}
R_K\bm U
:=
\sum_{i=1}^L
\Bigl(
U_i(x_K)\eta_i^{K}
+
\sum_{m=1}^d
\partial_m U_i(x_K)\eta_i^{K,(m)}
\Bigr)
\qquad\text{in }\omega(K).
\label{eq:bulk_reconstruction}
\end{equation}

Once the local bases have been constructed, we define the local homogenized
coefficients on $K$.
For $m,n=1,\ldots,d$ and $i,j=1,\ldots,L$, set
\begin{equation}
    \alpha_{ij,K}^{(mn)}
    :=
    \frac{1}{|\omega(K)|}
    \int_{\omega(K)}
    \kappa_\epsilon(y)
    \nabla_y\eta_i^{K,(m)}(y)
    \cdot
    \nabla_y\eta_j^{K,(n)}(y)
    \,\dif y,
    \label{eq:alpha_def}
\end{equation}
\begin{equation}
    \beta_{ij,K}^{(m)}
    :=
    \frac{1}{|\omega(K)|}
    \int_{\omega(K)}
    \kappa_\epsilon(y)
    \nabla_y\eta_i^{K,(m)}(y)
    \cdot
    \nabla_y\eta_j^K(y)
    \,\dif y,
    \label{eq:beta_def}
\end{equation}
and
\begin{equation}
    \gamma_{ij,K}
    :=
    \frac{1}{|\omega(K)|}
    \int_{\omega(K)}
    \kappa_\epsilon(y)
    \nabla_y\eta_i^K(y)
    \cdot
    \nabla_y\eta_j^K(y)
    \,\dif y.
    \label{eq:gamma_def}
\end{equation}
These coefficients are extended as piecewise constant fields on the coarse
grid:
\[
    \alpha_{ij}^{(mn)}(x)=\alpha_{ij,K}^{(mn)},\qquad
    \beta_{ij}^{(m)}(x)=\beta_{ij,K}^{(m)},\qquad
    \gamma_{ij}(x)=\gamma_{ij,K},
    \qquad x\in K.
\]
For convenience we also write
\[
    \beta_{ij}
    :=
    \bigl(\beta_{ij}^{(1)},\ldots,\beta_{ij}^{(d)}\bigr)^\top,
    \qquad
    \alpha_{ij}
    :=
    \bigl(\alpha_{ij}^{(mn)}\bigr)_{m,n=1}^d.
\]
Using the same reconstruction for trial and test fields, and scaling the RVE
energy density to the coarse block $K$, one obtains
\begin{align}
a_{\epsilon,K}(R_K\bm U,R_K\bm V)
&=
\int_K
\sum_{i=1}^L\sum_{j=1}^L
\Bigg[
    \gamma_{ij,K}U_iV_j
    +
    \sum_{m=1}^d
    \beta_{ij,K}^{(m)}\partial_m U_i\,V_j
    \notag\\
&\qquad\qquad
    +
    \sum_{n=1}^d
    \beta_{ji,K}^{(n)}U_i\,\partial_n V_j
    +
    \sum_{m=1}^d\sum_{n=1}^d
    \alpha_{ij,K}^{(mn)}
    \partial_m U_i\,\partial_n V_j
\Bigg]
\,\dif x.
\label{eq:local_bilinear_alpha_beta_gamma}
\end{align}
Summing over all coarse blocks yields
\begin{align}
a_\epsilon(u_\epsilon,v)
&\approx
\sum_{K\in\mathcal T_H}
a_{\epsilon,K}(R_K\bm U,R_K\bm V)
\notag\\
&=
\sum_{i=1}^L\sum_{j=1}^L
\int_\Omega
\Big[
    \gamma_{ij}U_iV_j
    +
    (\beta_{ij}\cdot\nabla U_i)V_j
\notag\\[-0.2em]
&\hspace{5em}
    +
    U_i(\beta_{ji}\cdot\nabla V_j)
    +
    (\alpha_{ij}\nabla U_i)\cdot\nabla V_j
\Big]
\,\dif x.
\label{eq:lhs_alpha_beta_gamma}
\end{align}

For a slowly varying source we retain the leading load coefficient
\[
    F_{j,K}^0
    :=
    \frac{1}{|\omega(K)|}
    \int_{\omega(K)}
    f(y)\eta_j^K(y)\,\dif y.
\]
The corresponding piecewise constant fields are
\[
    F_j^0(x)=F_{j,K}^0,
    \qquad x\in K,
\]
and the load is approximated by
\begin{equation}
    (f,v)
    \approx
    \sum_{j=1}^L
    \int_\Omega
        F_j^0V_j
    \,\dif x.
    \label{eq:rhs_alpha_beta_gamma}
\end{equation}

Combining \eqref{eq:lhs_alpha_beta_gamma} and
\eqref{eq:rhs_alpha_beta_gamma} yields the coarse weak formulation: find
$\bm U=(U_1,\ldots,U_L)^\top$ such that, for all
$\bm V=(V_1,\ldots,V_L)^\top$,
\begin{align}
&\sum_{i=1}^L\sum_{j=1}^L
\int_\Omega
\left[
    \gamma_{ij}U_iV_j
    +
    (\beta_{ij}\cdot\nabla U_i)V_j
\right]\,\dif x
\nonumber\\
&\quad+
\sum_{i=1}^L\sum_{j=1}^L
\int_\Omega
\left[
    U_i(\beta_{ji}\cdot\nabla V_j)
    +
    (\alpha_{ij}\nabla U_i)\cdot\nabla V_j
\right]\,\dif x
\nonumber\\
&\qquad=
\sum_{j=1}^L
\int_\Omega
    F_j^0V_j
\,\dif x.
\label{eq:coarse_weak_alpha_beta_gamma}
\end{align}
This bulk weak form supplies the interior contribution of I-MCHM.  The next
section constructs the additional interface bases and the corresponding
cut-cell and interface-segment terms that couple adjacent subdomains.

\section{Interface multicontinuum homogenization}
\label{sec:interface-mchm}

We now construct I-MCHM.  Interior coarse blocks are treated by the bulk
review in Section~\ref{sec:review_mchm}.  Across each material interface we
introduce two-sided constrained bases, extract bulk cut-cell and
interface-segment coefficients, assemble the global weak form over the
subdomain adjacency graph, and prove discrete well-posedness.
The remainder of the section is organized accordingly:
Section~\ref{sec:interface_expansion} constructs the local two-sided
reconstruction;
Section~\ref{sec:global_form} assembles the global bilinear form;
and Section~\ref{sec:discrete_stability} establishes discrete stability.

\subsection{Local reconstruction with subdomain-wise macroscopic variables}
\label{sec:interface_expansion}

We consider macroscopic continuum variables defined separately in each
subdomain.
For every $k=1,\ldots,N_\Omega$,
\[
    \bm U^{(k)}
    =
    \bigl(U_1^{(k)},\ldots,U_{L_k}^{(k)}\bigr)^\top
\]
collects the $H_\epsilon$-cell continuum averages associated with the
conductivity-based continua of $\Omega_k$.
These fields are continuous within their own subdomain, but they are not
required to match across a material interface: in general
$U_j^{(k)}\big|_{\Gamma_{k\ell}}$ need not equal any continuum field from
$\Omega_\ell$, and when $L_k\ne L_\ell$ no one-to-one pairing exists.
Accordingly, a local reconstruction near $\Gamma_{k\ell}$ must admit
independent macroscopic data from both sides.

The construction in this subsection is written for a single adjacent pair of
subdomains.  This is only a local simplification: the same two-sided
procedure applies independently to every interface edge
$(k,\ell)\in\mathcal E$, and the $N_\Omega$-subdomain model is obtained by
assembling these pairwise contributions over the adjacency graph as in
Section~\ref{sec:global_form}.

Fix an adjacent pair $(k,\ell)\in\mathcal E$, with interface
$\Gamma_{k\ell}=\partial\Omega_k\cap\partial\Omega_\ell$ as in
Section~\ref{sec:problem_setup}.  To simplify the formulas, we relabel the two
local sides as $a=1,2$; the assembly later maps these local labels back to the
global subdomain indices $k$ and $\ell$.
For this fixed pair we set $L_1=L_k$ and $L_2=L_\ell$, and identify
$\bm U^{(1)}=\bm U^{(k)}$ and $\bm U^{(2)}=\bm U^{(\ell)}$, with the same
identification for test variables.
Each continuum variable on a side is again the local average of the fine-scale
solution over a high- or low-permeability region of that subdomain on
observing cells of size $H_\epsilon$; on the interface neighborhood these
averages are treated as samples of slowly varying macroscopic fields and enter
the first-order reconstruction on $\omega(E)$.

Let $\mathcal T_H$ denote the coarse partition from Section~\ref{sec:mesh}, and write
\[
\mathcal T_H(\Gamma_{k\ell})
:=
\{K\in\mathcal T_H:\,|\overline K\cap\Gamma_{k\ell}|_{d-1}>0\}
\]
for the coarse blocks that meet the interface.
A coarse block $C\in\mathcal T_H(\Gamma_{k\ell})$ is an interface neighborhood
for the pair $(k,\ell)$ when its local sampling region contains
positive-volume portions of both subdomains.  We use
$E\in\mathcal T_H(\Gamma_{k\ell})$ as an index for such a neighborhood; thus
$E$ labels a coarse block rather than a lower-dimensional mesh edge.  The
interface RVE $\omega(E)\subset C$ is the union of the representative portions
sampled from the two sides of $\Gamma_{k\ell}$.  Its oversampled patch
$\omega^+(E)\supset\omega(E)$ is obtained by adjoining $s_{\mathrm{os}}$
layers of $H_\epsilon$-cells while retaining both sides of the interface.
Let $x_E$ be the centroid of $\omega(E)$.

Throughout this section, $x$ denotes a macroscopic position and $y$ the local
microscopic coordinate.
To avoid confusion with the fast variable $y$, the centers of the observing
cells are denoted by $z_q$, where $q$ is the observing-cell index.
Let $\mathcal I_E$ denote the observing cells contained in $\omega^+(E)$.
For a local side $a=1,2$, $j=1,\ldots,L_a$, and $q\in\mathcal I_E$, we
denote by
\[
    \psi_{j,q}^{a}(y)
\]
the characteristic function of the $j$-th continuum on that side within the
$q$-th observing cell.  Empty continuum-cell intersections are omitted.  Define
the zeroth- and first-order moments
\[
    s_{j,q}^{a}
    :=
    \int_{\omega^+(E)}
    \psi_{j,q}^{a}(y)\,\dif y,
\qquad
    r_{j,q}^{a,(m)}
    :=
    \int_{\omega^+(E)}
    (y_m-x_{E,m})\psi_{j,q}^{a}(y)\,\dif y,
\]
for $m=1,\ldots,d$.  We use the local energy space
$V(\omega^+(E)):=H_0^1(\omega^+(E))$.

\subsubsection{Two-sided constrained bases on the interface neighborhood}

For each local side $a=1,2$ and each continuum $i=1,\ldots,L_a$ of that side,
we construct constrained multiscale bases on the interface neighborhood that
reproduce the macroscopic continuum average associated with side $a$ and index
$i$.
The bases are computed with the true coefficient $\kappa_\epsilon$ on
$\omega^+(E)$ and are constrained by continuum moments from both sides of the
interface, so that independent macroscopic data from the adjacent subdomains
can enter the same local reconstruction.
Let
\[
    a_{\epsilon,E}^+(\eta,v)
    :=
    \int_{\omega^+(E)}
    \kappa_\epsilon(y)\nabla_y\eta(y)\cdot\nabla_y v(y)\,\dif y.
\]
The value basis $\eta_i^{a,E,+}\in V(\omega^+(E))$ and multipliers
$\{\mu_{ijq}^{a,b,E}\}$ are defined by the saddle-point problem
\[
\begin{aligned}
    a_{\epsilon,E}^+(\eta_i^{a,E,+},v)
    &=
    \sum_{b=1}^2\sum_{j=1}^{L_b}\sum_{q\in\mathcal I_E}
    \mu_{ijq}^{a,b,E}
    \int_{\omega^+(E)}
    \psi_{j,q}^{b}(y)v(y)\,\dif y,
    \\
    &\hspace{8em}\forall v\in V(\omega^+(E)),
\end{aligned}
\]
together with the constraints
\[
\begin{aligned}
    \int_{\omega^+(E)}
    \eta_i^{a,E,+}(y)\psi_{j,q}^{b}(y)\,\dif y
    &=\delta_{ab}\delta_{ij}s_{j,q}^{b},\\
    &\hspace{-8em}b=1,2,\quad j=1,\ldots,L_b,\quad q\in\mathcal I_E.
\end{aligned}
\]
Thus $\eta_i^{a,E,+}$ has unit average on the $i$-th continuum of side $a$ and
zero average on all other active continua of either side.

Similarly, for $m=1,\ldots,d$, the first-order basis
$\eta_i^{a,E,+,(m)}\in V(\omega^+(E))$ and multipliers
$\{\mu_{ijq}^{a,b,E,(m)}\}$ solve
\[
\begin{aligned}
    a_{\epsilon,E}^+(\eta_i^{a,E,+,(m)},v)
    &=
    \sum_{b=1}^2\sum_{j=1}^{L_b}\sum_{q\in\mathcal I_E}
    \mu_{ijq}^{a,b,E,(m)}
    \int_{\omega^+(E)}
    \psi_{j,q}^{b}(y)v(y)\,\dif y,
    \\
    &\hspace{8em}\forall v\in V(\omega^+(E)),
\end{aligned}
\]
together with
\[
\begin{aligned}
    \int_{\omega^+(E)}
    \eta_i^{a,E,+,(m)}(y)\psi_{j,q}^{b}(y)\,\dif y
    &=\delta_{ab}\delta_{ij}r_{j,q}^{b,(m)},\\
    &\hspace{-8em}b=1,2,\quad j=1,\ldots,L_b,\quad q\in\mathcal I_E.
\end{aligned}
\]

The basis functions used on the interface RVE $\omega(E)$ are the restrictions
\[
    \eta_i^{a,E}
    :=
    \eta_i^{a,E,+}|_{\omega(E)},
    \qquad
    \eta_i^{a,E,(m)}
    :=
    \eta_i^{a,E,+,(m)}|_{\omega(E)}.
\]

\subsubsection{Basis decomposition and localization}

To interpret the interface effect, we decompose each total basis into a bulk
reference response and a correction.

For each local side $a=1,2$, let $\widehat\kappa_\epsilon^{(a)}$ be a bulk
reference coefficient associated with that side on the interface neighborhood.
It extends the locally sampled pattern of side $a$ across $\omega^+(E)$ and is
used only for this interpretation; no global periodicity is required.

Define the corresponding reference bilinear form
\[
    \widehat a_{\epsilon,E}^{a,+}(\eta,v)
    :=
    \int_{\omega^+(E)}
    \widehat\kappa_\epsilon^{(a)}(y)
    \nabla_y\eta(y)\cdot\nabla_y v(y)\,\dif y.
\]
The reference bases $\widehat\eta_i^{a,E,+}$ and
$\widehat\eta_i^{a,E,+,(m)}$ solve the preceding value- and first-order
saddle-point problems with $a_{\epsilon,E}^+$ replaced by
$\widehat a_{\epsilon,E}^{a,+}$ and with identical moment constraints.
The bulk bases on $\omega(E)$ are the restrictions
\[
    \widehat\eta_i^{a,E}
    :=
    \widehat\eta_i^{a,E,+}|_{\omega(E)},
    \qquad
    \widehat\eta_i^{a,E,(m)}
    :=
    \widehat\eta_i^{a,E,+,(m)}|_{\omega(E)}.
\]

The interface-localized corrections are then
\[
    \eta_{i,\Gamma}^{a,E,+}
    :=
    \eta_i^{a,E,+}-\widehat\eta_i^{a,E,+},
\]
and
\[
    \eta_{i,\Gamma}^{a,E,+,(m)}
    :=
    \eta_i^{a,E,+,(m)}-\widehat\eta_i^{a,E,+,(m)},
\]
with restrictions
$\eta_{i,\Gamma}^{a,E}
=\eta_{i,\Gamma}^{a,E,+}|_{\omega(E)}$ and
$\eta_{i,\Gamma}^{a,E,(m)}
=\eta_{i,\Gamma}^{a,E,+,(m)}|_{\omega(E)}$.
By construction,
\[
\eta_i^{a,E}
=
\widehat\eta_i^{a,E}+\eta_{i,\Gamma}^{a,E},
\qquad
\eta_i^{a,E,(m)}
=
\widehat\eta_i^{a,E,(m)}+\eta_{i,\Gamma}^{a,E,(m)},
\]
so the full basis is the sum of a bulk part associated with the reference
medium of side $a$ and a correction that, in the sense of the decay estimate
below, concentrates near the interface.

\begin{lemma}[Zero constrained moments]
\label{lem:zero-interface-moments}
Fix a local side $a\in\{1,2\}$ and a continuum index $i\in\{1,\ldots,L_a\}$.
For every active side--continuum--cell triple $(b,j,q)$ on $\omega^+(E)$,
\[
\int_{\omega^+(E)}
\eta_{i,\Gamma}^{a,E,+}\psi_{j,q}^{b}\,\dif y=0.
\]
Likewise, for each first-order direction $m=1,\ldots,d$,
\[
\int_{\omega^+(E)}
\eta_{i,\Gamma}^{a,E,+,(m)}\psi_{j,q}^{b}\,\dif y=0.
\]
\end{lemma}
\begin{proof}
By construction, the two-sided continuum bases
$\eta_i^{a,E,+}$, $\eta_i^{a,E,+,(m)}$ and their bulk counterparts
$\widehat\eta_i^{a,E,+}$, $\widehat\eta_i^{a,E,+,(m)}$ are constrained by
the same family of moment functionals
$\{\mathcal M_{b,j,q}\}$ with identical right-hand sides
$\delta_{ab}\delta_{ij}s_{j,q}^{b}$ and
$\delta_{ab}\delta_{ij}r_{j,q}^{b,(m)}$, respectively.  Subtracting the
corresponding constraint equations therefore yields both identities for the
corrections
$\eta_{i,\Gamma}^{a,E,+}=\eta_i^{a,E,+}-\widehat\eta_i^{a,E,+}$ and
$\eta_{i,\Gamma}^{a,E,+,(m)}=\eta_i^{a,E,+,(m)}-\widehat\eta_i^{a,E,+,(m)}$.
\end{proof}
Thus the interface correction changes the microscopic energy profile of the
basis without altering any continuum moment imposed on \(\omega^+(E)\).
In particular, \(\eta_{i,\Gamma}^{a,E,+},\eta_{i,\Gamma}^{a,E,+,(m)}\) lie in
the kernel of the constraint map and cannot be detected by the macroscopic
continuum averages alone.

This zero-moment structure also explains why the correction localizes near the
interface, independently of whether the side-\(a\) pattern extension
\(\widehat\kappa_\epsilon^{(a)}\) coincides with \(\kappa_\epsilon\) on the
opposite side.  Let \(\mathcal M_E\) collect the active constraint
functionals
\[
 (\mathcal M_Ev)_{b,j,q}
 :=
 \int_{\omega^+(E)}v\,\psi_{j,q}^{b}\,\dif y .
\]
Lemma~\ref{lem:zero-interface-moments} places the corrections in
\(\ker\mathcal M_E\).  More precisely, for a basis associated with side \(a\)
and continuum \(i\):
\begin{itemize}
\item
on the opposite side \(b\neq a\), both \(\eta_i^{a,E,+}\) and
\(\widehat\eta_i^{a,E,+}\) have zero average on every continuum of that side
(and likewise for the first-order bases).  Each field therefore decays into
the opposite subdomain by the standard NLMC localization for zero-moment
continua~\cite{zhao2020analysis}, so their difference is small away from the
interface regardless of the pattern mismatch between \(\kappa_\epsilon\) and
\(\widehat\kappa_\epsilon^{(a)}\) there;
\item
on the matching side \(a\), the difference itself has zero average on every
continuum of that side, and the same NLMC cutoff argument yields decay away
from the interface into side \(a\).
\end{itemize}
Consequently, under the usual bounded stable-decomposition hypotheses for the
constrained spaces, there exist constants \(C_{\rm loc}>0\) and \(0<\rho<1\),
independent of the oversampling depth \(s_{\mathrm{os}}\), such that with
\[
\mathcal N_{s_{\mathrm{os}}}(\Gamma_{k\ell})
:=
\{y\in\omega^+(E):
\operatorname{dist}(y,\Gamma_{k\ell})<s_{\mathrm{os}}H_\epsilon\}
\]
one has
\begin{equation}
\begin{aligned}
&\int_{\omega^+(E)\setminus\mathcal N_{s_{\mathrm{os}}}(\Gamma_{k\ell})}
\kappa_\epsilon
\left(
|\nabla\eta_{i,\Gamma}^{a,E,+}|^2
+|\nabla\eta_{i,\Gamma}^{a,E,+,(m)}|^2
\right)\,\dif y
\\
&\hspace{5em}\leq
C_{\rm loc}\rho^{s_{\mathrm{os}}}
\int_{\omega^+(E)}
\kappa_\epsilon
\left(
|\nabla\eta_{i,\Gamma}^{a,E,+}|^2
+|\nabla\eta_{i,\Gamma}^{a,E,+,(m)}|^2
\right)\,\dif y.
\end{aligned}
\label{eq:interface_localization_decay}
\end{equation}
Estimate
\eqref{eq:interface_localization_decay} quantifies that the interface
correction concentrates near \(\Gamma_{k\ell}\).  We invoke the established
NLMC cutoff mechanism under the stable-decomposition hypotheses above and do
not reproduce the corresponding argument here.

For convenience, define the total local basis functions
\[
    \theta_i^{a,E}
    :=
    \widehat\eta_i^{a,E}+\eta_{i,\Gamma}^{a,E}
    =
    \eta_i^{a,E},
\]
and
\[
    \theta_i^{a,E,(m)}
    :=
    \widehat\eta_i^{a,E,(m)}
    +
    \eta_{i,\Gamma}^{a,E,(m)}
    =
    \eta_i^{a,E,(m)}.
\]
Thus each basis may be used directly or written as the sum of a bulk part and
an interface-localized correction.
In a local patch lying entirely inside one subdomain, the true coefficient and
the bulk reference coincide, so the correction vanishes and one recovers the
bulk basis of that subdomain.  More generally, under
\eqref{eq:interface_localization_decay}, the correction is negligible away from
the interface in the local energy norm.

\subsubsection{Local expansions and energy decomposition}

Using independent macroscopic data from both sides and the same local space
for trial and test fields, define
\begin{equation}
\begin{aligned}
R_E\bm U
&:=
\sum_{a=1}^2
\sum_{i=1}^{L_a}
\Bigl(
U_i^{(a)}(x_E)\theta_i^{a,E}
+
\sum_{m=1}^d
\partial_m U_i^{(a)}(x_E)\theta_i^{a,E,(m)}
\Bigr),
\\
\widehat R_E^{(a)}\bm U
&:=
\sum_{i=1}^{L_a}
\Bigl(
U_i^{(a)}(x_E)\widehat\eta_i^{a,E}
+
\sum_{m=1}^d
\partial_m U_i^{(a)}(x_E)\widehat\eta_i^{a,E,(m)}
\Bigr),
\\
\widehat R_E\bm U
&:=
\sum_{a=1}^2\widehat R_E^{(a)}\bm U,
\\
R_{\Gamma,E}\bm U
&:=
R_E\bm U-\widehat R_E\bm U,
\end{aligned}
\label{eq:interface_reconstruction}
\end{equation}
so that
\[
R_E\bm U
=
\widehat R_E\bm U+R_{\Gamma,E}\bm U,
\]
where $x_E$ is the centroid of $\omega(E)$.
Let $C_E\in\mathcal T_H$ be the physical coarse block associated with $E$,
with representative sampling region $\omega(E)\subset C_E$, and write
\[
\Gamma_E
:=
\overline{C_E}\cap\Gamma_{k\ell}
\]
for the portion of the fixed-pair interface $\Gamma_{k\ell}$ lying in $C_E$.
(Under the local side labels $a=1,2$, this is simply the interface segment
between the two sides inside $C_E$.)  The complete local reconstruction energy
on $C_E$ admits the exact algebraic decomposition
\begin{equation}
\begin{aligned}
a_{\epsilon,C_E}\bigl(R_E\bm U,R_E\bm V\bigr)
&=
\sum_{a=1}^2
a_{\epsilon,C_E\cap\Omega_a}
\bigl(\widehat R_E^{(a)}\bm U,\widehat R_E^{(a)}\bm V\bigr)
\\
&\qquad
+
\mathcal P_{C_E}^\Gamma(\bm U,\bm V),
\end{aligned}
\label{eq:energy_reference_correction}
\end{equation}
where the interface correction
\begin{equation}
\begin{aligned}
\mathcal P_{C_E}^\Gamma(\bm U,\bm V)
:={}&
a_{\epsilon,C_E}\bigl(R_E\bm U,R_E\bm V\bigr)
\\
&-
\sum_{a=1}^2
a_{\epsilon,C_E\cap\Omega_a}
\bigl(\widehat R_E^{(a)}\bm U,\widehat R_E^{(a)}\bm V\bigr)
\end{aligned}
\label{eq:interface_correction_form}
\end{equation}
collects every contribution that is not accounted for by the single-sided bulk
extensions on the cut cells $C_E\cap\Omega_a$.  Cross terms between distinct
sides vanish on each cut cell because the bulk reference reconstructions of
opposite sides are associated with disjoint media.  Under
\eqref{eq:interface_localization_decay}, Cauchy--Schwarz bounds
$\mathcal P_{C_E}^\Gamma$ by the decaying correction energy, so the correction
is an interface-localized modification of the bulk cut-cell energy rather than
a separate positive jump penalty.  The decomposition is an exact algebraic
interpretation of the two-sided reconstruction energy rather than an
additional algorithmic step.

\subsubsection{Bulk cut-cell and interface coefficients}

On each cut cell $C_E\cap\Omega_a$, the single-sided bulk reconstruction
$\widehat R_E^{(a)}$ yields bulk MCHM densities exactly as in
Section~\ref{sec:review_mchm}.  For $i,j=1,\ldots,L_a$ and $m,n=1,\ldots,d$,
denote the corresponding piecewise-constant coefficients by
$\alpha_{ij}^{(mn),a}$, $\beta_{ij}^{(m),a}$, and $\gamma_{ij}^{a}$, so that
\begin{equation}
\begin{aligned}
&a_{\epsilon,C_E\cap\Omega_a}
\bigl(\widehat R_E^{(a)}\bm U,\widehat R_E^{(a)}\bm V\bigr)
\\
&\quad=
\int_{C_E\cap\Omega_a}
\sum_{i,j=1}^{L_a}
\Big[
\gamma_{ij}^{a}U_i^{(a)}V_j^{(a)}
+
\bigl(\beta_{ij}^{a}\cdot\nabla U_i^{(a)}\bigr)V_j^{(a)}
\\
&\qquad\qquad
+
U_i^{(a)}\bigl(\beta_{ji}^{a}\cdot\nabla V_j^{(a)}\bigr)
+
\bigl(\alpha_{ij}^{a}\nabla U_i^{(a)}\bigr)\cdot\nabla V_j^{(a)}
\Big]
\,\dif x.
\end{aligned}
\label{eq:bulk_cutcell_form}
\end{equation}

The interface coefficients are defined so that the energy difference
\eqref{eq:interface_correction_form} is represented as a surface integral on
$\Gamma_E$.  For affine macroscopic data on $C_E$, there exist matrices
$\alpha_{ij}^{\Gamma,ab}=(\alpha_{ij}^{(mn),\Gamma,ab})_{m,n}$, vectors
$\beta_{ij}^{\Gamma,ab}$, and scalars $\gamma_{ij}^{\Gamma,ab}$
($a,b=1,2$, $i=1,\ldots,L_a$, $j=1,\ldots,L_b$) such that
\begin{equation}
\begin{aligned}
\mathcal P_{C_E}^\Gamma(\bm U,\bm V)
&=
\int_{\Gamma_E}
\sum_{a,b=1}^2
\sum_{i=1}^{L_a}
\sum_{j=1}^{L_b}
\Big[
\gamma_{ij}^{\Gamma,ab}
U_i^{(a)}V_j^{(b)}
+
\bigl(\beta_{ij}^{\Gamma,ab}\cdot\nabla U_i^{(a)}\bigr)
V_j^{(b)}
\\
&\qquad\qquad
+
U_i^{(a)}
\bigl(\beta_{ji}^{\Gamma,ba}\cdot\nabla V_j^{(b)}\bigr)
+
\bigl(\alpha_{ij}^{\Gamma,ab}\nabla U_i^{(a)}\bigr)
\cdot
\nabla V_j^{(b)}
\Big]
\,\dif S.
\end{aligned}
\label{eq:interface_line_form}
\end{equation}
Thus the interface coefficients are $(d-1)$-dimensional densities associated
with the energy difference between the two-sided reconstruction and the
single-sided bulk extensions; they are not the volume densities of the
complete local energy on $\omega(E)$.  Opposite-side blocks $a\neq b$ furnish
the effective continuum coupling across $\Gamma_{k\ell}$ and remain well
defined when $L_1\neq L_2$.
Combining \eqref{eq:energy_reference_correction}--\eqref{eq:interface_line_form}
gives the local identity
\begin{equation}
\begin{aligned}
&a_{\epsilon,C_E}\bigl(R_E\bm U,R_E\bm V\bigr)
\\
&\quad=
\sum_{a=1}^2
a_{\epsilon,C_E\cap\Omega_a}
\bigl(\widehat R_E^{(a)}\bm U,\widehat R_E^{(a)}\bm V\bigr)
\\
&\qquad+
\int_{\Gamma_E}
\sum_{a,b=1}^2
\sum_{i=1}^{L_a}
\sum_{j=1}^{L_b}
\Big[
\gamma_{ij}^{\Gamma,ab}
U_i^{(a)}V_j^{(b)}
+
\bigl(\beta_{ij}^{\Gamma,ab}\cdot\nabla U_i^{(a)}\bigr)
V_j^{(b)}
\\
&\qquad\qquad\qquad
+
U_i^{(a)}
\bigl(\beta_{ji}^{\Gamma,ba}\cdot\nabla V_j^{(b)}\bigr)
+
\bigl(\alpha_{ij}^{\Gamma,ab}\nabla U_i^{(a)}\bigr)
\cdot
\nabla V_j^{(b)}
\Big]
\,\dif S.
\end{aligned}
\label{eq:local_cutcell_gamma_identity}
\end{equation}

\subsection{Global bulk--interface weak formulation}
\label{sec:global_form}

We now assemble the local models without introducing a separate pointwise
transmission law for the continuum variables.  Recall that
$\mathcal T_H(\Gamma_{k\ell})$ denotes the coarse blocks that meet
$\Gamma_{k\ell}$.
On a coarse block $K$ contained in the interior of $\Omega_k$, the bulk
coefficients of Section~\ref{sec:review_mchm} define the contribution
\begin{equation}
\begin{aligned}
a_H^{k,K}(\bm U^{(k)},\bm V^{(k)})
&=
\int_{K}
\sum_{i,j=1}^{L_k}
\Big[
\gamma_{ij}^{k}U_i^{(k)}V_j^{(k)}
+
\bigl(\beta_{ij}^{k}\cdot\nabla U_i^{(k)}\bigr)V_j^{(k)}
\\
&\qquad\qquad
+
U_i^{(k)}\bigl(\beta_{ji}^{k}\cdot\nabla V_j^{(k)}\bigr)
+
\bigl(\alpha_{ij}^{k}\nabla U_i^{(k)}\bigr)\cdot\nabla V_j^{(k)}
\Big]
\,\dif x.
\end{aligned}
\label{eq:bulk_block_form}
\end{equation}
On a block $K\in\mathcal T_H(\Gamma_{k\ell})$, map the local sides $a=1,2$ to
the global pair $(k,\ell)$ and insert the bulk cut-cell densities of
\eqref{eq:bulk_cutcell_form} together with the interface line densities of
\eqref{eq:interface_line_form}.  The resulting local contribution is
\begin{equation}
\begin{aligned}
&a_H^{k\ell,K}(\bm U,\bm V)
\\
&\quad=
\sum_{a\in\{k,\ell\}}
\int_{K\cap\Omega_a}
\sum_{i,j=1}^{L_a}
\Big[
\gamma_{ij}^{a}U_i^{(a)}V_j^{(a)}
+
\bigl(\beta_{ij}^{a}\cdot\nabla U_i^{(a)}\bigr)V_j^{(a)}
\\
&\qquad\qquad\qquad
+
U_i^{(a)}\bigl(\beta_{ji}^{a}\cdot\nabla V_j^{(a)}\bigr)
+
\bigl(\alpha_{ij}^{a}\nabla U_i^{(a)}\bigr)\cdot\nabla V_j^{(a)}
\Big]
\,\dif x
\\
&\qquad+
\int_{K\cap\Gamma_{k\ell}}
\sum_{a,b\in\{k,\ell\}}
\sum_{i=1}^{L_a}
\sum_{j=1}^{L_b}
\Big[
\gamma_{ij}^{\Gamma,ab}
U_i^{(a)}V_j^{(b)}
+
\bigl(\beta_{ij}^{\Gamma,ab}\cdot\nabla U_i^{(a)}\bigr)
V_j^{(b)}
\\
&\qquad\qquad\qquad
+
U_i^{(a)}
\bigl(\beta_{ji}^{\Gamma,ba}\cdot\nabla V_j^{(b)}\bigr)
+
\bigl(\alpha_{ij}^{\Gamma,ab}\nabla U_i^{(a)}\bigr)
\cdot
\nabla V_j^{(b)}
\Big]
\,\dif S.
\end{aligned}
\label{eq:iface_block_form}
\end{equation}
The global bilinear form is therefore
\begin{equation}
\begin{aligned}
a_H(\bm U,\bm V)
:={}&
\sum_{k=1}^{N_\Omega}
\sum_{\substack{K\in\mathcal T_H\\ K\subset\Omega_k}}
a_H^{k,K}(\bm U^{(k)},\bm V^{(k)})
\\
&+
\sum_{(k,\ell)\in\mathcal E}
\sum_{K\in\mathcal T_H(\Gamma_{k\ell})}
a_H^{k\ell,K}\!\left(
(\bm U^{(k)},\bm U^{(\ell)}),
(\bm V^{(k)},\bm V^{(\ell)})
\right).
\label{eq:global_pairwise_bilinear}
\end{aligned}
\end{equation}
Interface blocks replace, rather than duplicate, the corresponding interior
bulk contributions: each coarse block is assembled once, either as a purely
bulk element or as a cut-cell element with an interface segment.
Near a junction, blocks belonging to different adjacent pairs may meet; the
model superposes the pairwise interface integrals on the corresponding
segments rather than introducing a separate codimension-two contribution.
Because the same local reconstruction is used for trial and test fields,
\eqref{eq:global_pairwise_bilinear} inherits symmetry from the fine-scale
energy.  Under \eqref{eq:interface_localization_decay}, as a neighborhood moves
into the interior of a subdomain the correction vanishes and
\eqref{eq:iface_block_form} reduces to the standard bulk MCHM contribution.

By the local identity \eqref{eq:local_cutcell_gamma_identity},
\begin{equation}
a_H(\bm U,\bm V)
=
\sum_{K\in\mathcal T_H}
a_{\epsilon,K}\bigl(R_K\bm U,R_K\bm V\bigr),
\label{eq:AH_full_energy_equivalence}
\end{equation}
where $R_K$ denotes the bulk reconstruction of Section~\ref{sec:review_mchm}
on interior blocks and the two-sided reconstruction $R_E$ on interface blocks.
Thus the cut-cell plus interface-segment form
\eqref{eq:global_pairwise_bilinear} coincides with the reconstruction-energy
representation used in the stability argument below.

The coarse load uses the same reconstructions.  With interior contributions
$F_K^k$ as in Section~\ref{sec:review_mchm} and interface contributions
\[
F_K^{k\ell}(\bm V^{(k)},\bm V^{(\ell)})
:=
\int_K f(y)\,R_E\bm V(y)\,\dif y,
\]
we set
\[
F_H(\bm V)
:=
\sum_{k=1}^{N_\Omega}
\sum_{\substack{K\in\mathcal T_H\\ K\subset\Omega_k}}
F_K^k(\bm V^{(k)})
+
\sum_{(k,\ell)\in\mathcal E}
\sum_{K\in\mathcal T_H(\Gamma_{k\ell})}
F_K^{k\ell}(\bm V^{(k)},\bm V^{(\ell)}).
\]
Let $V_H^k$ be the space of continuous piecewise-linear finite-element
functions on the restriction of $\mathcal T_H$ to $\Omega_k$, with zero trace
on $\partial\Omega\cap\partial\Omega_k$, and define
\[
\mathcal X_H:=\prod_{k=1}^{N_\Omega}[V_H^k]^{L_k}.
\]
The interface MCHM problem is to find $\bm U_H\in\mathcal X_H$ such that
\begin{equation}
a_H(\bm U_H,\bm V_H)=F_H(\bm V_H)
\qquad\forall\bm V_H\in\mathcal X_H.
\label{eq:global_pairwise_weak}
\end{equation}
Equation~\eqref{eq:global_pairwise_weak} is the primary coarse model used in
this work.  We write $A_H:=a_H$ for the associated bilinear form on
$\mathcal X_H$, and $\mathbf U_H$ for the coefficient vector of $\bm U_H$
after elimination of Dirichlet degrees of freedom.

For several subdomains the second sum in \eqref{eq:global_pairwise_bilinear}
runs over the adjacency graph $\mathcal E$.  Each edge corresponds to a
two-sided interface problem, while edges incident on a common subdomain share
that subdomain's continuum variables.  In the three-subdomain experiment of
Section~\ref{sec:nume}, one interface-segment contribution is associated with
each of the three interface branches.  At a triple junction the junction point has
codimension two, so no additional point contribution is introduced.
Deriving an independent junction law would require a separate local asymptotic
analysis and is outside the scope of this paper.

\subsection{Discrete stability}
\label{sec:discrete_stability}

We next establish well-posedness of the assembled coarse problem.  This is a
stability result for the finite-dimensional I-MCHM operator; it does not
require Galerkin orthogonality with the fine-scale problem.

By \eqref{eq:AH_full_energy_equivalence},
\begin{equation}
 A_H(\bm U,\bm V)
 =
 \sum_{K\in\mathcal T_H}
 a_{\epsilon,K}
       (R_K\bm U,R_K\bm V),
\label{eq:AH_local_energy_representation}
\end{equation}
where $R_K$ is the bulk or two-sided interface reconstruction on $K$.
On interface blocks, $R_K=R_E$ is the complete reconstruction.  The
interface line densities in \eqref{eq:interface_line_form} need not define a
positive form by themselves, yet they do not destroy nonnegativity of the
complete local energy in \eqref{eq:AH_local_energy_representation}.

We use the following standard discrete unisolvence condition.
\begin{assumption}[Energy unisolvence]
\label{ass:energy_unisolvence}
After removal of empty continuum constraints and elimination of Dirichlet
coarse degrees of freedom, the local reconstruction energies separate the
coarse fields:
\[
 \nabla R_K\bm V=0\ \text{a.e. in }K
 \quad\text{for every }K\in\mathcal T_H
 \qquad\Longrightarrow\qquad
 \bm V=0\ \text{in }\mathcal X_H.
\]
\end{assumption}
This condition states that the combined local energy seminorms have trivial
kernel on the Dirichlet-reduced coarse space.  It is equivalent to positive
definiteness of the assembled stiffness matrix.  Independence of the retained
constraints, unisolvent sampling of the coarse P1 fields, and elimination of
boundary null modes are necessary structural checks; in the reported
computations the condition is verified algebraically after assembly.

\begin{theorem}[Symmetry, energy-norm coercivity, and discrete stability]
\label{thm:discrete_stability}
Suppose $\kappa_\epsilon>0$ almost everywhere and Assumption
\ref{ass:energy_unisolvence} holds.  Then
\[
 \|\bm V\|_{A_H}:=A_H(\bm V,\bm V)^{1/2}
\]
is a norm on $\mathcal X_H$, and
\begin{align}
 A_H(\bm U,\bm V)&=A_H(\bm V,\bm U),
 \label{eq:AH_symmetry}\\
 |A_H(\bm U,\bm V)|
 &\leq \|\bm U\|_{A_H}\|\bm V\|_{A_H},
 \label{eq:AH_continuity}\\
 A_H(\bm V,\bm V)&=\|\bm V\|_{A_H}^2.
 \label{eq:AH_coercivity}
\end{align}
Hence \eqref{eq:global_pairwise_weak} has a unique solution and satisfies
\begin{equation}
 \|\bm U_H\|_{A_H}
 \leq
 \|F_H\|_{\mathcal X_H'},
 \qquad
 \|F_H\|_{\mathcal X_H'}
 :=
 \sup_{\bm V\in\mathcal X_H\setminus\{0\}}
 \frac{|F_H(\bm V)|}{\|\bm V\|_{A_H}}.
\label{eq:discrete_stability_bound}
\end{equation}
\end{theorem}

\begin{proof}
Symmetry follows from \eqref{eq:AH_local_energy_representation} and symmetry
of every fine-scale energy form $a_{\epsilon,K}$.  Positivity of
$\kappa_\epsilon$ gives
\[
 A_H(\bm V,\bm V)
 =
 \sum_{K\in\mathcal T_H}
 \int_K\kappa_\epsilon
 |\nabla R_K\bm V|^2\,\dif y
 \geq0.
\]
If this quantity is zero, every nonnegative summand vanishes; energy
unisolvence then implies $\bm V=0$.  Thus $\|\cdot\|_{A_H}$ is a norm.
Applying Cauchy--Schwarz first to each local energy and then to the finite sum
proves \eqref{eq:AH_continuity}; \eqref{eq:AH_coercivity} follows from the
definition of the norm.  The Lax--Milgram theorem, equivalently positive
definiteness of the finite-dimensional stiffness matrix, gives existence and
uniqueness.  Finally, choosing $\bm V=\bm U_H$ in
\eqref{eq:global_pairwise_weak} yields
\[
 \|\bm U_H\|_{A_H}^2
 =F_H(\bm U_H)
 \leq
 \|F_H\|_{\mathcal X_H'}\|\bm U_H\|_{A_H},
\]
which proves \eqref{eq:discrete_stability_bound}.
\end{proof}

\begin{remark}[Dependence of the stability constant]
\label{rem:stability_constant}
Theorem~\ref{thm:discrete_stability} establishes well-posedness in the natural
discrete energy norm $\|\cdot\|_{A_H}$, but does not provide a contrast-, $H$-,
or $H_\epsilon$-uniform equivalence with an unweighted coarse $H^1$ norm.
Since $\mathcal X_H$ is finite dimensional, coercivity with constant one in
$\|\cdot\|_{A_H}$ also implies
\[
 A_H(\bm V,\bm V)\geq
 \lambda_{\min}(\mathsf K_H)\|\bm v\|_{\ell^2}^2
\]
for the coefficient vector $\bm v$ and the Dirichlet-reduced stiffness matrix
$\mathsf K_H$.  No claim is made that $\lambda_{\min}(\mathsf K_H)$ itself is
uniform in the contrast, $H$, or $H_\epsilon$.  Such a parameter-uniform
estimate would require quantitative stable-decomposition and reconstruction
bounds beyond the discrete well-posedness proved here.
\end{remark}

\section{Numerical Results}
\label{sec:nume}

We present numerical tests that assess I-MCHM on three geometries of increasing
complexity: straight interfaces with continuum-coupling ablations
(Section~\ref{sec:case1}), a curved interface (Section~\ref{sec:case2}), and a
three-subdomain triple junction (Section~\ref{sec:case3}).
The primary metric is a continuum-wise local-average error against fine-grid
solutions of the same problem.  The continua are the high- and low-permeability
regions (one continuum for a homogeneous subdomain).  We use
\[
    \kappa_{\rm high}=1,
    \qquad
    \kappa_{\rm low}=H_\epsilon^2/1000,
\]
with $h=H_\epsilon/N_{\rm sub}$ and $N_{\rm sub}=8$ unless otherwise stated.
All refinement studies set $\epsilon=H_\epsilon$ and let the contrast vary with
this scale.  Each row therefore compares a member of an $\epsilon$-indexed PDE
family rather than mesh convergence for one fixed coefficient field.  In
particular, decreasing $H_\epsilon=\epsilon$ at fixed coarse $H$ only drives the
error toward a residual controlled by $H$; the primary evidence is therefore a
coupled refinement in which the ratio $H/H_\epsilon$ is held fixed and chosen
sufficiently large that the observing scale resolves the local continua.
Throughout the reported experiments we take the representative value
$H/H_\epsilon=4$ (equivalently $H_\epsilon=\epsilon=H/4$); the precise integer
ratio is not singled out by the theory.

Let $K_\alpha^\epsilon$ be an $H_\epsilon$ observing cell and let
$\omega_{\alpha,j}^{(k)}$ denote the part of this cell belonging to continuum
$j$ of subdomain $\Omega_k$.
The averaging operator is
\[
\bigl(\Pi_\epsilon w\bigr)_{\alpha,j}^{(k)}
:=
\frac{1}{|\omega_{\alpha,j}^{(k)}|}
\int_{\omega_{\alpha,j}^{(k)}} w(x)\,\dif x
\quad\text{whenever }|\omega_{\alpha,j}^{(k)}|>0,
\]
with the corresponding volume-weighted $\ell^2$ norm.  The reported errors
measure numerical upscaling at fixed $\epsilon$ and do not separately estimate
the homogenization remainder associated with an $\epsilon\to0$ limit.

For the fine finite-element solution $u_h\approx u_\epsilon$ on the mesh of
size $h$ and the I-MCHM continuum fields
$\bm U_H=(U_j^{(k)})$, define
\[
    \overline u_{\alpha,j}^{(k),h}
    =
    \frac{1}{|\omega_{\alpha,j}^{(k)}|}
    \int_{\omega_{\alpha,j}^{(k)}} u_h(x)\,\dif x,
    \qquad
    \overline U_{\alpha,j}^{(k),H}
    =
    \frac{1}{|\omega_{\alpha,j}^{(k)}|}
    \int_{\omega_{\alpha,j}^{(k)}} U_j^{(k)}(x)\,\dif x .
\]
The homogenized local-average relative error is
\[
    E_{\rm loc}^{\rm Hom}
    =
    \left(
    \frac{
    \sum_{\alpha,k,j} |\omega_{\alpha,j}^{(k)}|
    \left|\overline U_{\alpha,j}^{(k),H}
    -
    \overline u_{\alpha,j}^{(k),h}\right|^2}
    {
    \sum_{\alpha,k,j} |\omega_{\alpha,j}^{(k)}|
    \left|\overline u_{\alpha,j}^{(k),h}\right|^2}
    \right)^{1/2}.
\]
Only positive-volume continua are included.  Restricting the same quotient to
one pair $(k,j)$ yields $E_j^{(k),\rm Hom}$; tables abbreviate this quantity by
$E_j^{(k)}$, and we write $E_{\rm all}=E_{\rm loc}^{\rm Hom}$.

For the straight-interface family in Table~\ref{tab:case1_convergence} we also
report a relative energy error based on the same local oversampling
reconstruction used to assemble the homogenized energy.  On each coarse
neighborhood $K$, let $R_K\bm U_H$ be the constrained reconstruction of the
coarse continuum solution on the oversampled patch, restricted to $K$.
Writing $a_{\epsilon,K}$ for the fine energy form integrated on $K$,
\[
    E_A^{\rm Hom}
    =
    \frac{
    \Bigl(
    \sum_{K}
    a_{\epsilon,K}
    \bigl(R_K\bm U_H-u_h,\,R_K\bm U_H-u_h\bigr)
    \Bigr)^{1/2}
    }{\|u_h\|_{a_\epsilon}},
    \qquad
    \|v\|_{a_\epsilon}=\sqrt{a_\epsilon(v,v)}.
\]
The remaining tables report only the local-average errors defined above.

\subsection{Case 1: straight-line interfaces and continuum-coupling ablations}
\label{sec:case1}

The first group comprises two straight interfaces.  The first subcase uses the
vertical interface $x=0.5$, and the second uses the antidiagonal interface
$x+y=1$.  In both subcases the two subdomains contain different channelized
media, and each subdomain is represented by two continua.  The sine-source
tests employ the weighted source
\[
    f_s(x,y)=\kappa(x,y)\sin(\pi x)\sin(\pi y).
\]

\subsubsection{Subcase 1a: interface \texorpdfstring{$x=0.5$}{x=0.5}}

For the vertical interface, the left and right halves have different channel
orientations.  At $H=1/7$ and $H_\epsilon=1/28$, the local-average error is
$E_{\rm loc}^{\rm Hom}=8.30\%$.  Along the same coupled path with
$H_\epsilon=\epsilon=H/4$, the geometry yields
$5.43\%$ at $H=1/9$, $H_\epsilon=1/36$;
$3.84\%$ at $H=1/11$, $H_\epsilon=1/44$; and
$E_{\rm loc}^{\rm Hom}=2.87\%$ at $H=1/13$, $H_\epsilon=1/52$.

\subsubsection{Subcase 1b: interface \texorpdfstring{$x+y=1$}{x+y=1}}

For the antidiagonal interface, the observing grid produces an aligned
staircase.  At $H=1/5$ and $H_\epsilon=1/20$, the local-average error is
$12.91\%$.  For the same geometry, the refined cases $H=1/9$,
$H_\epsilon=1/36$ and $H=1/13$, $H_\epsilon=1/52$ give errors of $6.36\%$ and
$4.34\%$, respectively.  Thus the numerical upscaling error decreases as $H$
and $\epsilon=H_\epsilon$ decrease together at this fixed observing-to-coarse
ratio.

Table~\ref{tab:case1_convergence} summarizes the coupled refinement path with
$H_\epsilon=\epsilon=H/4$ for the two straight-interface subcases.  Each
row compares I-MCHM with the fine solution for the same member of the
$\epsilon$-dependent problem family.

\begin{table}[ht]
\centering
\scriptsize
\caption{Numerical upscaling error for the straight-line interface cases along a coupled refinement path with sufficiently large fixed ratio $H/H_\epsilon$ (reported with $H_\epsilon=\epsilon=H/4$).  Each row compares I-MCHM with the fine solution for the same member of the $\epsilon$-dependent problem family.  The columns $E_j^{(k)}$ report the local-average error for continuum $U_j^{(k)}$; $E_{\rm all}$ is the volume-weighted aggregate $E_{\rm loc}^{\rm Hom}$. The column $E_A^{\rm Hom}$ is the relative energy error of the local oversampling Hom reconstruction. Columns labeled dofs report $\dim\mathcal X_H$ after Dirichlet elimination; the per-continuum columns use the abbreviation $E_j^{(k)}\equiv E_j^{(k),\rm Hom}$.}
\resizebox{\textwidth}{!}{%
\begin{tabular}{cccccccccc}
\hline
Interface & $H$ & $H_\epsilon$ & dofs & $E_{\rm all}$ & $E_A^{\rm Hom}$ & $E_1^{(1)}$ & $E_2^{(1)}$ & $E_1^{(2)}$ & $E_2^{(2)}$ \\
\hline
$x=0.5$ & $1/5$ & $1/20$ & 48 & 14.26\% & 37.52\% & 17.27\% & 17.48\% & 11.09\% & 11.32\% \\
$x=0.5$ & $1/7$ & $1/28$ & 96 & 8.30\% & 28.74\% & 10.38\% & 10.50\% & 6.01\% & 6.24\% \\
$x=0.5$ & $1/9$ & $1/36$ & 160 & 5.43\% & 23.34\% & 6.88\% & 6.96\% & 3.79\% & 4.01\% \\
$x=0.5$ & $1/11$ & $1/44$ & 240 & 3.84\% & 19.70\% & 4.89\% & 4.94\% & 2.63\% & 2.84\% \\
$x=0.5$ & $1/13$ & $1/52$ & 336 & 2.87\% & 17.07\% & 3.64\% & 3.68\% & 1.96\% & 2.16\% \\
$x+y=1$ & $1/5$ & $1/20$ & 52 & 12.91\% & 33.98\% & 10.68\% & 11.06\% & 14.38\% & 14.10\% \\
$x+y=1$ & $1/7$ & $1/28$ & 104 & 7.50\% & 26.89\% & 7.15\% & 7.31\% & 7.69\% & 7.70\% \\
$x+y=1$ & $1/9$ & $1/36$ & 172 & 6.36\% & 25.45\% & 6.18\% & 6.31\% & 6.47\% & 6.35\% \\
$x+y=1$ & $1/11$ & $1/44$ & 256 & 5.28\% & 22.70\% & 4.55\% & 4.55\% & 5.83\% & 5.64\% \\
$x+y=1$ & $1/13$ & $1/52$ & 356 & 4.34\% & 20.56\% & 3.70\% & 3.63\% & 4.84\% & 4.66\% \\
\hline
\end{tabular}}
\label{tab:case1_convergence}
\end{table}

\begin{figure}[ht]
    \centering
    \includegraphics[width=0.78\textwidth]{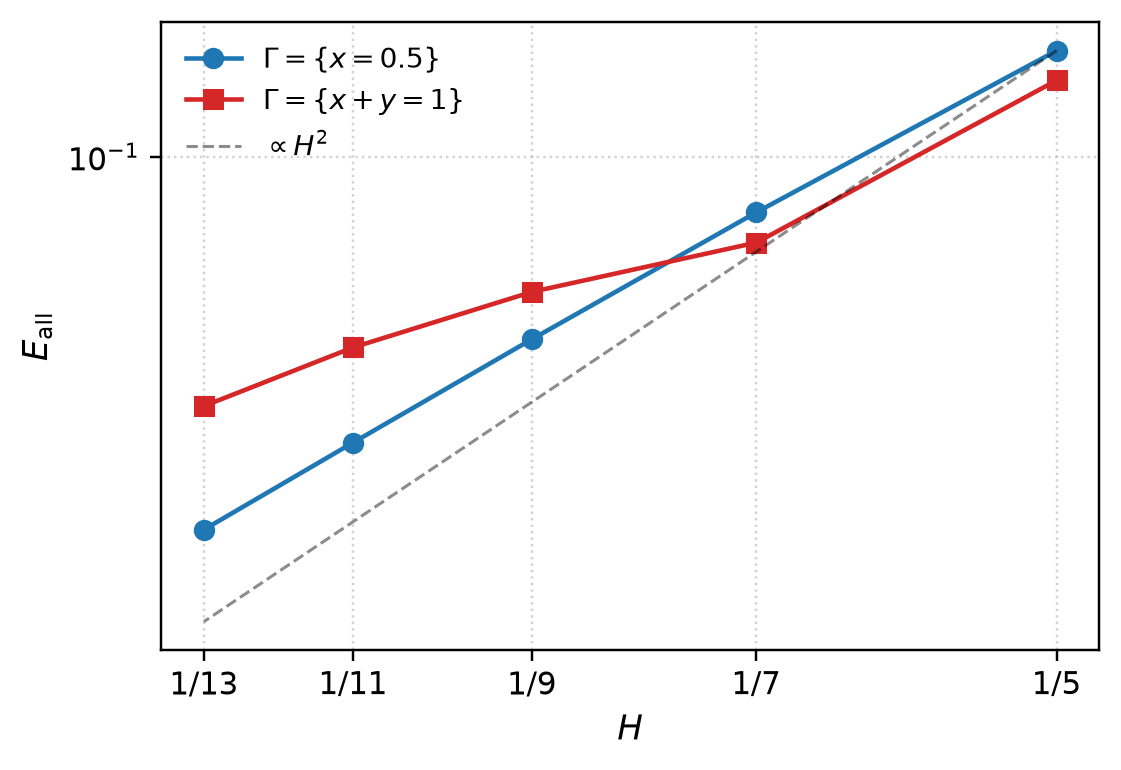}
    \caption{Aggregate local-average error $E_{\rm all}$ for the two straight-interface
    families along the coupled path with $H_\epsilon=\epsilon=H/4$.  The dashed
    line indicates an $H^{2}$ reference slope.}
    \label{fig:case1_Eall_vs_H}
\end{figure}

Figure~\ref{fig:case1_Eall_vs_H} displays the same aggregate errors against
$H$.  With a fixed, sufficiently large observing-to-coarse ratio, $E_{\rm all}$
decreases steadily as $H$ is refined for both interfaces.
Along this path, $E_A^{\rm Hom}$ decreases from $37.52\%$ to $17.07\%$ for the
vertical family and from $33.98\%$ to $20.56\%$ for the antidiagonal family,
remaining larger than $E_{\rm all}$, as expected for an energy comparison.

By contrast, refining only the microscopic/observing scale at fixed coarse $H$
does not produce the same decay.  For the vertical interface with $H=1/5$,
decreasing $H_\epsilon=\epsilon$ from $H$ to $H/16$ reduces $E_{\rm all}$ from
$21.83\%$ to about $14.4\%$, after which the error saturates.  The residual
level is controlled by the coarse-grid resolution $H$ (and by the
$\epsilon$-dependence of the PDE family), so further reduction of
$H_\epsilon=\epsilon$ alone cannot eliminate the coarse approximation error.
Accordingly, the primary convergence evidence reported below is coupled
refinement at a fixed, sufficiently large $H/H_\epsilon$, rather than a
fixed-$H$ observing-scale sweep.

\subsubsection{Continuum-coupling ablations}

We next isolate the interface-specific reconstruction from the subdomain bulk
models.  Both bulk-only baselines employ the same local energy assembly as
I-MCHM, except that an interface-cut coarse element is treated by same-side bulk
patches for the two materials: each contribution is integrated only over its
part of the cut element, and all cross-material blocks are omitted.
The discontinuous baseline leaves the two coarse traces independent; its
variational problem therefore induces a natural zero-flux condition on each
side.  The continuous baseline instead imposes
$U_j^{(1)}=U_j^{(2)}$ along the interface for the two matching continuum
indices, while allowing the normal derivatives to differ.

Table~\ref{tab:case1_interface_ablation} reports both a fixed-$H$ observing-scale
test and the first three points of the coupled path with
$H_\epsilon=\epsilon=H/4$.  For the axis-aligned interface, the two bulk-only
variants remain close, and I-MCHM is slightly more accurate at the smaller
tested $H$.  The antidiagonal case behaves differently: the independent-trace
error grows from $21.97\%$ to $45.78\%$ as $H$ decreases, whereas the
continuous-trace error decreases from $12.52\%$ to $4.21\%$; I-MCHM decreases
from $12.91\%$ to $6.36\%$.  Thus the failure of the discontinuous bulk model
is associated with its zero-flux closure, while a prescribed zero-jump closure
is highly effective when the continua on the two sides admit a one-to-one
correspondence and share the same continuum meaning on the interface.
This continuous baseline is not a general replacement for the present interface
construction: it assumes matching continuum labels and a pointwise transmission
law, whereas I-MCHM permits different numbers and meanings of continua and
introduces their coupling variationally.

\begin{table}[ht]
\centering
\caption{Bulk-only ablation over several $H$ and $H_\epsilon$.  Both bulk-only baselines omit all cross-material blocks.  The discontinuous version leaves the two subdomain traces independent (natural zero flux), whereas the continuous version enforces zero jump between matching continuum traces while allowing the normal derivatives to differ.}
\begin{tabular}{lccccc}
\hline
Interface & $H$ & $H_\epsilon$ & bulk-only & bulk-only & I-MCHM \\
 & & & discontinuous & continuous & \\
\hline
$x=0.5$ & $1/5$ & $1/5$ & 16.12\% & 16.64\% & 21.83\% \\
$x=0.5$ & $1/5$ & $1/10$ & 15.01\% & 15.86\% & 13.92\% \\
$x=0.5$ & $1/5$ & $1/20$ & 14.61\% & 15.46\% & 14.26\% \\
$x=0.5$ & $1/7$ & $1/28$ & 8.51\% & 9.25\% & 8.30\% \\
$x=0.5$ & $1/9$ & $1/36$ & 5.59\% & 6.35\% & 5.43\% \\
$x+y=1$ & $1/5$ & $1/20$ & 21.97\% & 12.52\% & 12.91\% \\
$x+y=1$ & $1/7$ & $1/28$ & 36.60\% & 6.76\% & 7.50\% \\
$x+y=1$ & $1/9$ & $1/36$ & 45.78\% & 4.21\% & 6.36\% \\
\hline
\end{tabular}
\label{tab:case1_interface_ablation}
\end{table}

\subsubsection{Unequal continuum counts}

To test the setting in which continua cannot be paired, we retain the
vertical interface $x=0.5$ but replace the right-hand pattern by the
homogeneous medium $\kappa\equiv1$.  The right subdomain therefore retains a
single continuum, while the left subdomain keeps the two-continuum channel
pattern of Case~1.  The source is the unscaled Gaussian
\[
f(x,y)
=
\exp\!\left(
-\frac{(x-0.4)^2+(y-0.5)^2}{2(0.15)^2}
\right),
\]
centered near the interface so that a nontrivial response is transmitted into
the homogeneous side.  For the continuous bulk-only baseline we identify every
continuum degree of freedom attached to the same interface node,
\[
U^{\mathrm{right}}
=
U_1^{\mathrm{left}}
=
U_2^{\mathrm{left}}
\quad\text{on each interior interface node,}
\]
rather than assuming a one-to-one continuum pairing.

Table~\ref{tab:case1_asymmetric_ablation} reports the coupled path with
$H_\epsilon=\epsilon=H/4$ for $H=1/5,1/7,1/9$.  Along this path the
I-MCHM aggregate error decreases from $13.33\%$ to $4.00\%$, while the
discontinuous and continuous bulk-only closures decrease only from about
$20\%$--$21\%$ to about $11\%$--$12\%$.  At every tested $H$, I-MCHM remains
substantially more accurate.  In particular, the continuous shared-node
constraint does not recover the accuracy of I-MCHM: it over-constrains the
transmitted right-hand continuum and also degrades the left high-conductivity
continuum.  Thus, when the adjacent subdomains have unequal continuum
structures, the two-sided interface reconstruction is essential.

\begin{table}[ht]
\centering
\caption{Unequal-continuum bulk-only ablation for the vertical interface $x=0.5$ along a coupled path with $H_\epsilon=\epsilon=H/4$.  The right subdomain is homogeneous with $\kappa\equiv1$ (one continuum), while the left subdomain retains the two-continuum channel pattern.  The source is the unscaled Gaussian centered at $(0.4,0.5)$ with standard deviation $0.15$.  The continuous bulk-only baseline identifies every continuum degree of freedom on the same interface node.  Errors are continuum local averages against the fine reference.}
\begin{tabular}{ccccc}
\hline
$H$ & $H_\epsilon$ & bulk-only & bulk-only & I-MCHM \\
 & & discontinuous & continuous & \\
\hline
$1/5$ & $1/20$ & 19.73\% & 21.40\% & 13.33\% \\
$1/7$ & $1/28$ & 13.55\% & 13.99\% & 6.18\% \\
$1/9$ & $1/36$ & 11.10\% & 11.58\% & 4.00\% \\
\hline
\end{tabular}
\label{tab:case1_asymmetric_ablation}
\end{table}

\FloatBarrier
\subsection{Case 2: curved interface}
\label{sec:case2}

The second case uses the circular interface
\[
    (x-0.5)^2+(y-0.5)^2=0.2^2 .
\]
The medium inside the circle is denoted by $\Omega_2$, and the medium outside
by $\Omega_1$.  Curved and oblique interfaces are represented on the observing
grid by staircase geometry: the smooth circle is assigned to $H_\epsilon$ cells
by evaluating the level set at cell centers, so the continuum variables see a
staircase approximation of the circle.  Geometric approximation error is
therefore not separated from upscaling error.  We use the weighted sine source
$f_s(x,y)=\kappa(x,y)\sin(\pi x)\sin(\pi y)$.

Figure~\ref{fig:curved_case} shows the curved-interface medium and the
fine-grid solution for $H=1/7$, $H_\epsilon=1/28$.  The local-average error is
$E_{\rm all}=7.93\%$.  Along the coupled path with
$H_\epsilon=\epsilon=H/4$, this source yields errors of $4.90\%$ at
$H=1/9$, $H_\epsilon=1/36$ and $5.25\%$ at $H=1/11$, $H_\epsilon=1/44$.

\begin{figure}[ht]
    \centering
    \begin{subfigure}{0.48\textwidth}
        \centering
        \includegraphics[width=\textwidth]{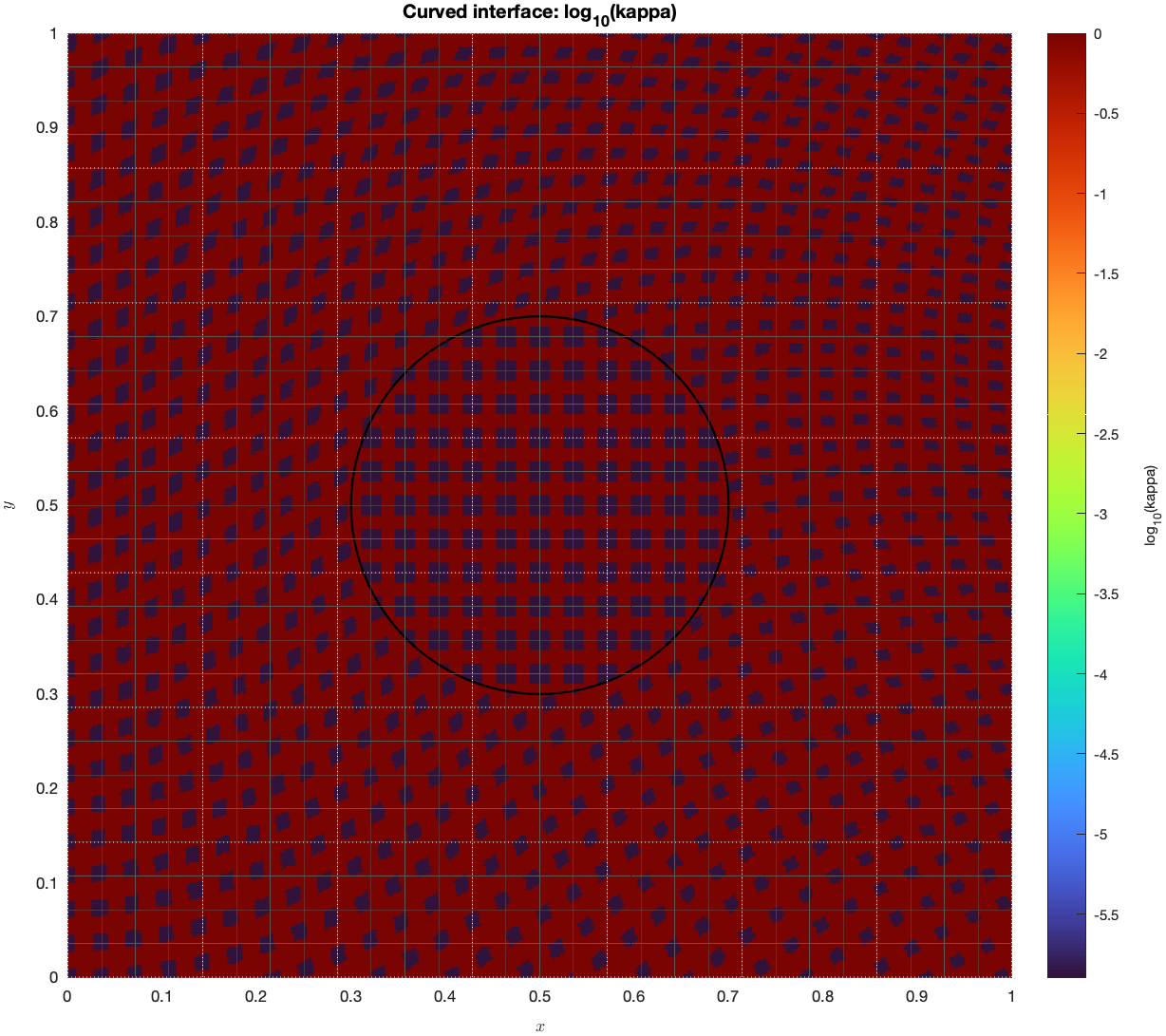}
        \caption{Medium coefficient.}
    \end{subfigure}
    \hfill
    \begin{subfigure}{0.48\textwidth}
        \centering
        \includegraphics[width=\textwidth]{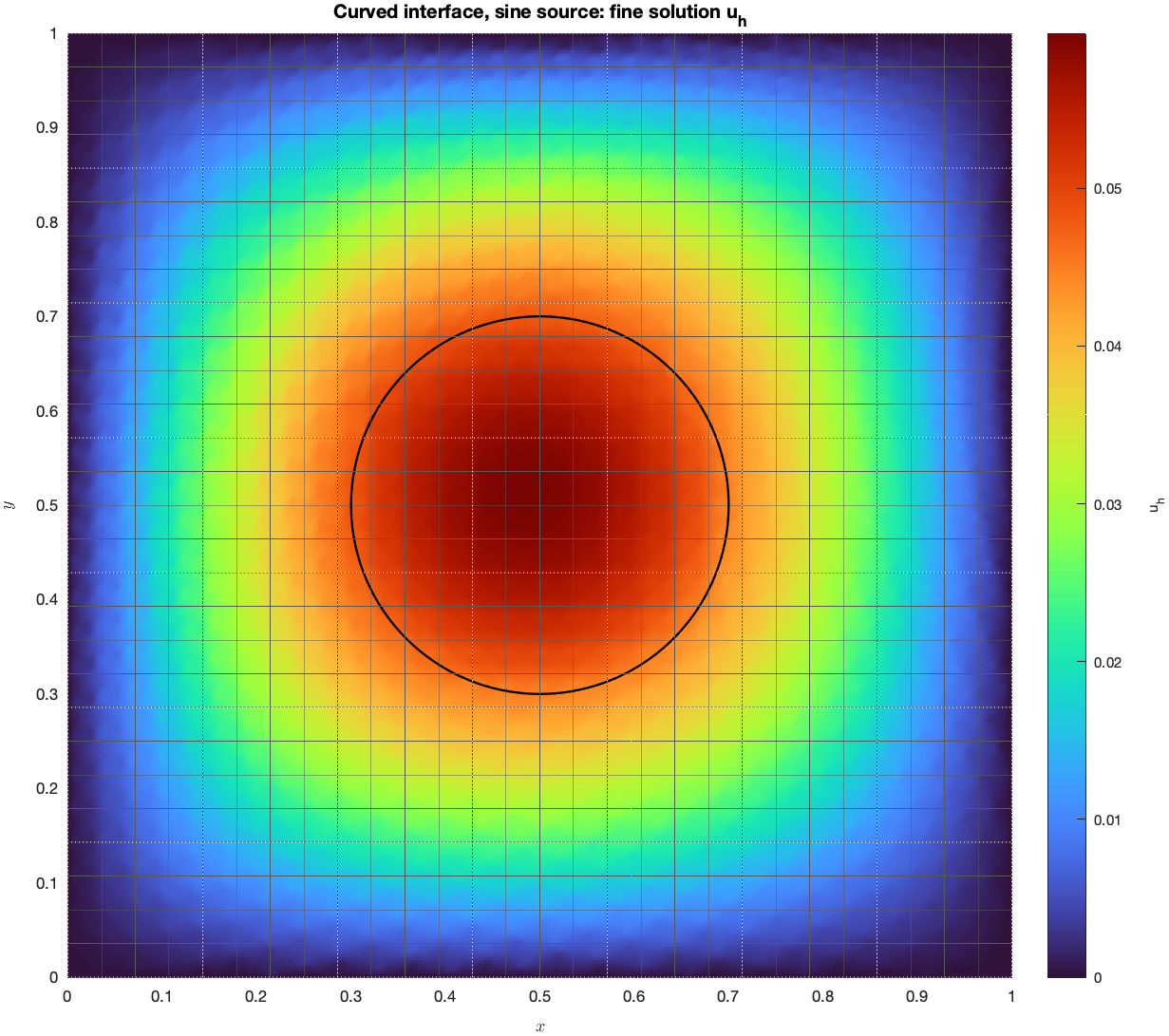}
        \caption{Fine-grid solution.}
    \end{subfigure}
    \caption{Curved-interface case with $H=1/7$, $H_\epsilon=1/28$, and
    $N_{\rm sub}=10$.}
    \label{fig:curved_case}
\end{figure}

Table~\ref{tab:case2_convergence} reports the coupled refinement path with
$H_\epsilon=\epsilon=H/4$ for the curved-interface sine source.

\begin{table}[ht]
\centering
\scriptsize
\caption{Numerical upscaling error for the curved-interface sine-source case along a coupled refinement path with $H_\epsilon=\epsilon=H/4$.  Each row compares I-MCHM with the fine solution for the same member of the $\epsilon$-dependent problem family.  The columns $E_j^{(k)}$ report the local-average error for continuum $U_j^{(k)}$; $E_{\rm all}$ is the volume-weighted aggregate $E_{\rm loc}^{\rm Hom}$. Columns labeled dofs report $\dim\mathcal X_H$ after Dirichlet elimination; the per-continuum columns use the abbreviation $E_j^{(k)}\equiv E_j^{(k),\rm Hom}$.}
\resizebox{\textwidth}{!}{%
\begin{tabular}{cccccccc}
\hline
$H$ & $H_\epsilon$ & dofs & $E_{\rm all}$ & $E_1^{(1)}$ & $E_2^{(1)}$ & $E_1^{(2)}$ & $E_2^{(2)}$ \\
\hline
$1/5$ & $1/20$ & 60 & 9.41\% & 9.59\% & 9.21\% & 9.77\% & 9.49\% \\
$1/7$ & $1/28$ & 100 & 7.93\% & 8.21\% & 7.22\% & 8.85\% & 8.47\% \\
$1/9$ & $1/36$ & 180 & 4.90\% & 4.95\% & 4.88\% & 5.52\% & 4.54\% \\
$1/11$ & $1/44$ & 260 & 5.25\% & 7.75\% & 4.71\% & 5.10\% & 4.96\% \\
\hline
\end{tabular}}
\label{tab:case2_convergence}
\end{table}

\FloatBarrier
\subsection{Case 3: three subdomains with a triple junction}
\label{sec:case3}

The third case comprises three subdomains.  The bottom triangular subdomain is
separated from the top part by the two line segments joining $(0,0)$ to
$(0.5,0.5)$ and $(0.5,0.5)$ to $(1,0)$.  The top part is split into left and
right subdomains by the vertical segment $x=0.5$ for $y>0.5$.  This geometry
therefore tests three interface segments meeting at a triple junction, together with
a spatially varying number of active continua.  The source is a Gaussian
centered at $(0.5,0.25)$ with standard deviation $0.15$, supported only on the
high-conductivity continua,
\[
\begin{aligned}
    f(x,y)&=
    g(x,y)\,
    \mathbf{1}_{\{\kappa(x,y)>
    \sqrt{\kappa_{\mathrm{low}}\kappa_{\mathrm{high}}}\}},
    \\[-0.25em]
    g(x,y)&=
    \exp\!\left(
    -\frac{(x-0.5)^2+(y-0.25)^2}{2(0.15)^2}
    \right).
\end{aligned}
\]

Figure~\ref{fig:three_subdomain_case} shows the three-subdomain coefficient
field and fine-grid solution for $H=1/7$, $H_\epsilon=1/28$, and
$N_{\rm sub}=8$.  This experiment tests the adjacency-graph formulation with
three coupled pairwise interface reconstructions.  Their contributions share
the global continuum unknowns associated with each subdomain, rather than
forming independent two-subdomain models.  No additional point law is imposed
at the triple junction.
For this representative case, the homogenized local-average error is
$E_{\rm loc}^{\rm Hom}=9.54\%$.

\begin{figure}[ht]
    \centering
    \begin{subfigure}{0.48\textwidth}
        \centering
        \includegraphics[width=\textwidth]{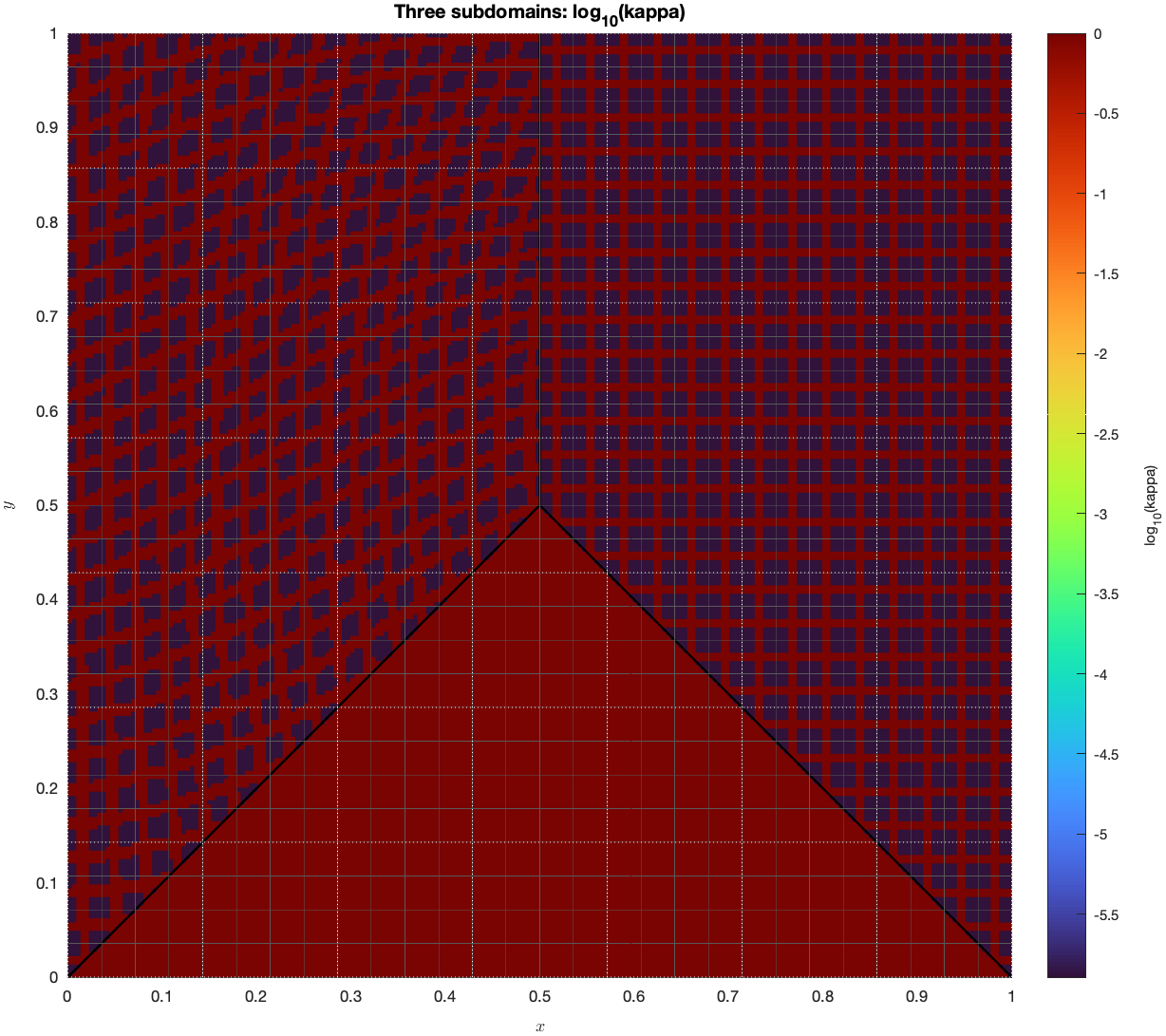}
        \caption{Medium coefficient.}
    \end{subfigure}
    \hfill
    \begin{subfigure}{0.48\textwidth}
        \centering
        \includegraphics[width=\textwidth]{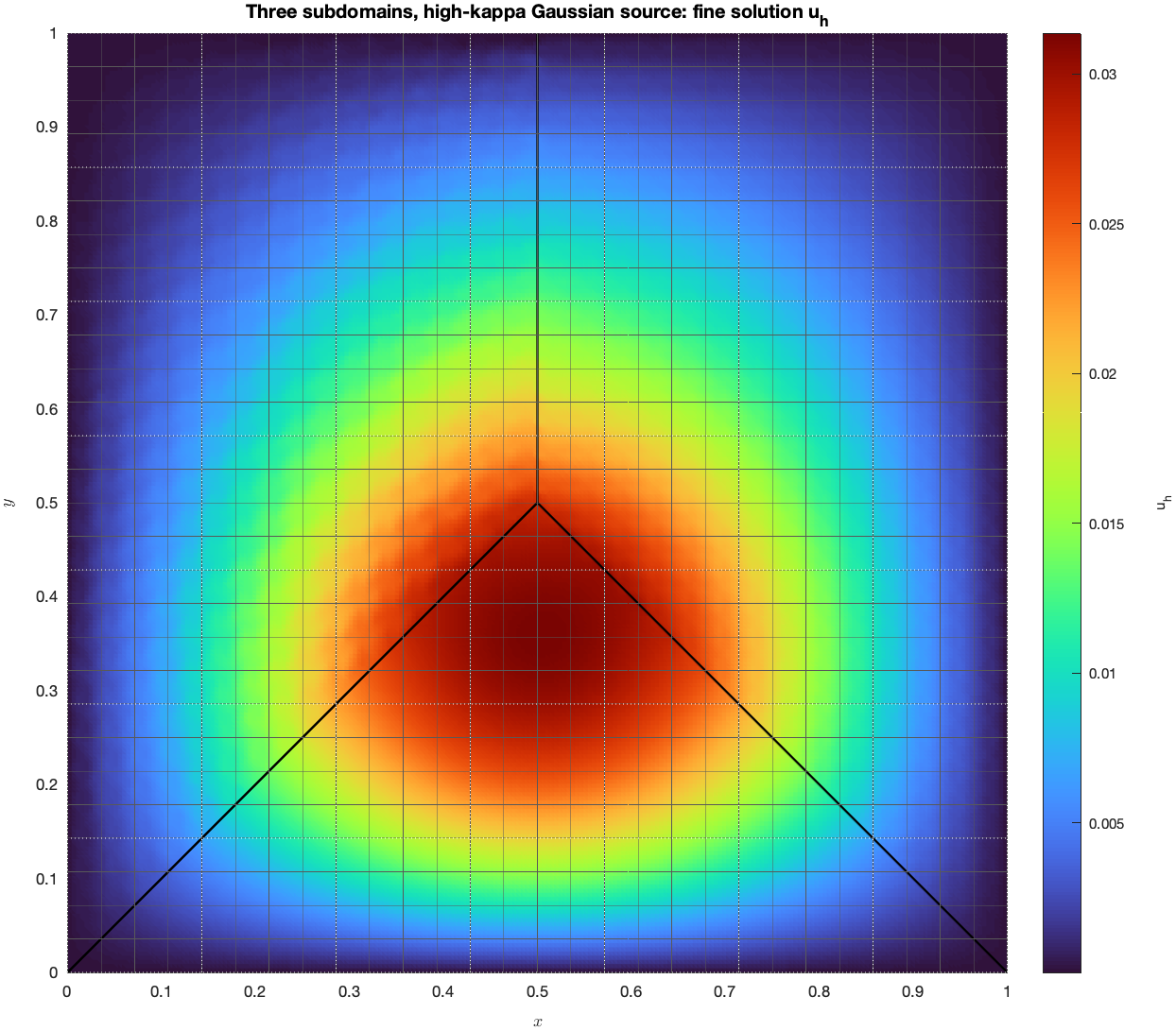}
        \caption{Fine-grid solution.}
    \end{subfigure}
    \caption{Three-subdomain case with $H=1/7$, $H_\epsilon=1/28$,
    $N_{\rm sub}=8$, and high-conductivity-supported Gaussian source centered
    at $(0.5,0.25)$ with standard deviation $0.15$.}
    \label{fig:three_subdomain_case}
\end{figure}

Table~\ref{tab:case3_convergence} shows the coupled refinement path with
$H_\epsilon=\epsilon=H/4$ for the three-subdomain high-conductivity-supported
Gaussian source.  Because every two-continuum cell portion contains both
phases, all five component errors are defined in every row.  The aggregate
local-average error decreases from $13.47\%$ at $H=1/5$ to $7.65\%$ at
$H=1/11$.
The published three-subdomain path terminates at $H=1/11$: a further $H=1/13$
run retained both continua in every cell portion, but one portion had only a
$1/16$ phase fraction and produced a large nonmonotone spike in $E_{11}$.

\begin{table}[ht]
\centering
\scriptsize
\caption{Numerical upscaling error for the three-subdomain case with high-conductivity-supported Gaussian source ($f=0$ on low-$\kappa$ continua; $\sigma=0.15$) along a coupled refinement path with $H_\epsilon=\epsilon=H/4$.  Each row compares I-MCHM with the fine solution for the same member of the $\epsilon$-dependent high-contrast problem family.  The columns $E_j^{(k)}$ report the local-average error for continuum $U_j^{(k)}$; $E_{\rm all}$ is the volume-weighted aggregate $E_{\rm loc}^{\rm Hom}$. Columns labeled dofs report $\dim\mathcal X_H$ after Dirichlet elimination; the per-continuum columns use the abbreviation $E_j^{(k)}\equiv E_j^{(k),\rm Hom}$.}
\resizebox{\textwidth}{!}{%
\begin{tabular}{ccccccccc}
\hline
$H$ & $H_\epsilon$ & dofs & $E_{\rm all}$ & $E_1^{(1)}$ & $E_2^{(1)}$ & $E_1^{(2)}$ & $E_2^{(2)}$ & $E_1^{(3)}$ \\
\hline
$1/5$ & $1/20$ & 48 & 13.47\% & 13.72\% & 13.85\% & 7.56\% & 7.78\% & 15.61\% \\
$1/7$ & $1/28$ & 93 & 9.54\% & 10.20\% & 10.36\% & 5.50\% & 5.59\% & 10.76\% \\
$1/9$ & $1/36$ & 152 & 8.37\% & 9.09\% & 9.25\% & 5.96\% & 6.03\% & 9.02\% \\
$1/11$ & $1/44$ & 225 & 7.65\% & 8.39\% & 8.67\% & 5.58\% & 5.65\% & 8.11\% \\
\hline
\end{tabular}}
\label{tab:case3_convergence}
\end{table}

Overall, the straight-interface tests exhibit the most stable local-average
behavior along the coupled $H$--$\epsilon$ refinement paths.  The
unequal-continuum vertical test further shows that I-MCHM remains effective
when the two sides retain different numbers of continua: its aggregate error
decreases from $13.33\%$ at $H=1/5$ to $4.00\%$ at $H=1/9$, while both
bulk-only closures stay above $11\%$.  The curved-interface sine source remains
comparable to the straight cases at the smaller tested scales.  The
three-subdomain path remains below $10\%$ through $H=1/11$ on this
multiple-interface geometry.

\FloatBarrier
\section{Conclusion}
\label{sec:conclu}

We have developed I-MCHM as a variational interface closure for multicontinuum
homogenization.  Relative to bulk MCHM, which determines continuum equations
inside each subdomain, the distinguishing ingredients are two-sided constrained
bases on interface neighborhoods and a coarse bilinear form assembled from bulk
cut-cell integrals together with interface-segment corrections over the
subdomain adjacency graph.  Adjacent subdomains may therefore retain different
continuum counts and meanings without a prescribed pointwise transmission law.
A bulk-reference plus zero-moment correction decomposition identifies the
interface contribution as a localized modification of the bilinear form and
explains recovery of the standard bulk basis away from the interface.  Under
discrete energy unisolvence, the assembled form is symmetric, continuous, and
coercive in its reconstruction-energy norm, so the coarse problem is uniquely
solvable and satisfies the a priori bound \eqref{eq:discrete_stability_bound}.

The numerical experiments support this construction along a coupled path with
a sufficiently large fixed ratio $H/H_\epsilon$ (reported with
$H_\epsilon=\epsilon=H/4$): aggregate continuum-average errors for the
oblique, curved-sine, and three-subdomain tests lie in the range
$7.50$--$9.54\%$ at $H=1/7$ (vertical straight: $8.30\%$).  The two straight
families reach $2.87\%$ and $4.34\%$ at $H=1/13$, and the three-subdomain path
reaches $7.65\%$ at $H=1/11$.  Ablations show why the interface term matters:
independent zero-flux bulk models deteriorate on the oblique interface, while a
zero-jump constraint works when continua can be paired but fails for unequal
continuum structures, where I-MCHM reduces the aggregate error to $4.00\%$ at
$H=1/9$ compared with about $11\%$--$14\%$ for either bulk-only closure.


\section*{CRediT authorship contribution statement}
\textbf{Wing Tat Leung:} Writing -- review \& editing, Writing -- original draft,
Methodology, Investigation, Formal analysis, Conceptualization.
\textbf{Zhihang Xu:} Writing -- review \& editing, Writing -- original draft,
Methodology, Investigation, Conceptualization.

\section*{Funding}
Wing Tat Leung is partially supported by the Hong Kong RGC Early Career Scheme
21307223.

\section*{Declaration of competing interest}
The authors declare that they have no known competing financial interests or
personal relationships that could have appeared to influence the work reported
in this paper.

\section*{Data availability}
No data was used for the research described in the article.

\section*{Acknowledgments}
Wing Tat Leung is partially supported by the Hong Kong RGC Early Career Scheme
21307223.

\section*{Declaration of generative AI and AI-assisted technologies}
During the preparation of this work, the authors used ChatGPT to improve
readability and language polish. After using this tool, the authors reviewed
and edited the content as needed and take full responsibility for the content
of the publication.

\bibliographystyle{elsarticle-num}
\bibliography{references}

@article{hou1997multiscale,
  author = {Hou, Thomas Y. and Wu, Xiao-Hui},
  title = {A multiscale finite element method for elliptic problems in composite materials and porous media},
  journal = {Journal of Computational Physics},
  volume = {134},
  number = {1},
  pages = {169--189},
  year = {1997},
  doi = {10.1006/jcph.1997.5682}
}

@book{efendiev2009multiscale,
  author = {Efendiev, Yalchin and Hou, Thomas Y.},
  title = {Multiscale Finite Element Methods: Theory and Applications},
  series = {Surveys and Tutorials in the Applied Mathematical Sciences},
  volume = {4},
  publisher = {Springer},
  address = {New York},
  year = {2009},
  doi = {10.1007/978-0-387-09496-0}
}

@article{chung2018non,
  author = {Chung, Eric T. and Efendiev, Yalchin and Leung, Wing Tat and Vasilyeva, Maria and Wang, Yating},
  title = {Non-local multi-continua upscaling for flows in heterogeneous fractured media},
  journal = {Journal of Computational Physics},
  volume = {372},
  pages = {22--34},
  year = {2018},
  doi = {10.1016/j.jcp.2018.05.038}
}

@article{zhao2020analysis,
  author = {Zhao, Lina and Chung, Eric T.},
  title = {An analysis of the {NLMC} upscaling method for high contrast problems},
  journal = {Journal of Computational and Applied Mathematics},
  volume = {367},
  pages = {112480},
  year = {2020},
  doi = {10.1016/j.cam.2019.112480}
}

@article{efendiev2022multicontinuum,
  author = {Efendiev, Yalchin and Leung, Wing Tat},
  title = {Multicontinuum homogenization and its relation to nonlocal multicontinuum theories},
  journal = {Journal of Computational Physics},
  volume = {474},
  pages = {111761},
  year = {2023},
  doi = {10.1016/j.jcp.2022.111761}
}

@book{milton2022theory,
  author = {Milton, Graeme W.},
  title = {The Theory of Composites},
  publisher = {Society for Industrial and Applied Mathematics},
  address = {Philadelphia},
  year = {2022}
}

@book{bear2013dynamics,
  author = {Bear, Jacob},
  title = {Dynamics of Fluids in Porous Media},
  publisher = {Dover Publications},
  address = {Mineola, New York},
  year = {2013}
}

@article{berre2018flow,
  author = {Berre, Inga and Doster, Florian and Keilegavlen, Eirik},
  title = {Flow in fractured porous media: A review of conceptual models and discretization approaches},
  journal = {Transport in Porous Media},
  volume = {130},
  pages = {215--236},
  year = {2019},
  doi = {10.1007/s11242-018-1171-6}
}

@article{engquist2008asymptotic,
  author = {Engquist, Bj{\"o}rn and Souganidis, Panagiotis E.},
  title = {Asymptotic and numerical homogenization},
  journal = {Acta Numerica},
  volume = {17},
  pages = {147--190},
  year = {2008},
  doi = {10.1017/S0962492904000015}
}

@book{bensoussan1978asymptotic,
  author = {Bensoussan, Alain and Lions, Jacques-Louis and Papanicolaou, George},
  title = {Asymptotic Analysis for Periodic Structures},
  series = {Studies in Mathematics and its Applications},
  volume = {5},
  publisher = {North-Holland},
  address = {Amsterdam},
  year = {1978}
}

@article{allaire1992homogenization,
  author = {Allaire, Gr{\'e}goire},
  title = {Homogenization and two-scale convergence},
  journal = {SIAM Journal on Mathematical Analysis},
  volume = {23},
  number = {6},
  pages = {1482--1518},
  year = {1992},
  doi = {10.1137/0523084}
}

@article{barenblatt1960basic,
  author = {Barenblatt, G. I. and Zheltov, Iu. P. and Kochina, I. N.},
  title = {Basic concepts in the theory of seepage of homogeneous liquids in fissured rocks [strata]},
  journal = {Journal of Applied Mathematics and Mechanics},
  volume = {24},
  number = {5},
  pages = {1286--1303},
  year = {1960},
  doi = {10.1016/0021-8928(60)90107-6}
}

@article{arbogast1990derivation,
  author = {Arbogast, Todd and Douglas, Jr., Jim and Hornung, Ulrich},
  title = {Derivation of the double porosity model of single phase flow via homogenization theory},
  journal = {SIAM Journal on Mathematical Analysis},
  volume = {21},
  number = {4},
  pages = {823--836},
  year = {1990},
  doi = {10.1137/0521046}
}

@article{e2003heterogeneous,
  author = {E, Weinan and Engquist, Bj{\"o}rn},
  title = {The heterogeneous multiscale methods},
  journal = {Communications in Mathematical Sciences},
  volume = {1},
  number = {1},
  pages = {87--132},
  year = {2003},
  doi = {10.4310/CMS.2003.v1.n1.a8}
}

@article{abdulle2012heterogeneous,
  author = {Abdulle, Assyr and E, Weinan and Engquist, Bj{\"o}rn and Vanden-Eijnden, Eric},
  title = {The heterogeneous multiscale method},
  journal = {Acta Numerica},
  volume = {21},
  pages = {1--87},
  year = {2012},
  doi = {10.1017/S0962492912000025}
}

@article{efendiev2013generalized,
  author = {Efendiev, Yalchin and Galvis, Juan and Hou, Thomas Y.},
  title = {Generalized multiscale finite element methods ({GMsFEM})},
  journal = {Journal of Computational Physics},
  volume = {251},
  pages = {116--135},
  year = {2013},
  doi = {10.1016/j.jcp.2013.04.045}
}

@article{henning2014localized,
  author = {Henning, Patrick and M{\aa}lqvist, Axel},
  title = {Localized orthogonal decomposition techniques for boundary value problems},
  journal = {SIAM Journal on Scientific Computing},
  volume = {36},
  number = {4},
  pages = {A1609--A1634},
  year = {2014},
  doi = {10.1137/130933038}
}

@article{chung2024multicontinuum,
  author = {Chung, Eric and Efendiev, Yalchin and Galvis, Juan and Leung, Wing Tat},
  title = {Multicontinuum homogenization. {G}eneral theory and applications},
  journal = {Journal of Computational Physics},
  volume = {510},
  pages = {112980},
  year = {2024},
  doi = {10.1016/j.jcp.2024.112980}
}

@article{bunoiu2019upscaling,
  author = {Bunoiu, Renata and Timofte, Claudia},
  title = {Upscaling of a diffusion problem with interfacial flux jump leading to a modified {B}arenblatt model},
  journal = {Zeitschrift f{\"u}r Angewandte Mathematik und Mechanik},
  volume = {99},
  number = {2},
  pages = {e201800018},
  year = {2019},
  doi = {10.1002/zamm.201800018}
}

@article{chen1998finite,
  author = {Chen, Zhiming and Zou, Jun},
  title = {Finite element methods and their convergence for elliptic and parabolic interface problems},
  journal = {Numerische Mathematik},
  volume = {79},
  number = {2},
  pages = {175--202},
  year = {1998},
  doi = {10.1007/s002110050336}
}

@article{leveque1994immersed,
  author = {LeVeque, Randall J. and Li, Zhilin},
  title = {The immersed interface method for elliptic equations with discontinuous coefficients and singular sources},
  journal = {SIAM Journal on Numerical Analysis},
  volume = {31},
  number = {4},
  pages = {1019--1044},
  year = {1994},
  doi = {10.1137/0731054}
}

@article{chu2010multiscale,
  author = {Chu, C.-C. and Graham, I. G. and Hou, T.-Y.},
  title = {A new multiscale finite element method for high-contrast elliptic interface problems},
  journal = {Mathematics of Computation},
  volume = {79},
  number = {272},
  pages = {1915--1955},
  year = {2010},
  doi = {10.1090/S0025-5718-2010-02372-5}
}

@article{hou2018iteratively,
  author = {Hou, Thomas Y. and Hwang, Feng-Nan and Liu, Pengfei and Yao, Chien-Chou},
  title = {An iteratively adaptive multiscale finite element method for elliptic interface problems},
  journal = {Applied Numerical Mathematics},
  volume = {127},
  pages = {211--225},
  year = {2018},
  doi = {10.1016/j.apnum.2018.01.009}
}

@article{efendiev2011multiscale,
  author = {Efendiev, Yalchin and Galvis, Juan and Wu, Xiao-Hui},
  title = {Multiscale finite element methods for high-contrast problems using local spectral basis functions},
  journal = {Journal of Computational Physics},
  volume = {230},
  number = {4},
  pages = {937--955},
  year = {2011},
  doi = {10.1016/j.jcp.2010.09.026}
}
\end{document}